\documentclass[leqno]{article}

\usepackage{amsmath,amssymb,amsthm,bm,mathrsfs}
\newcommand{\R}{\mathbb{R}}
\newcommand{\N}{\mathbb{N}}
\newcommand{\Z}{\mathbb{Z}}

\newcommand{\PP}{\mathbb{P}}

\newcommand{\Div}{\mathrm{div}\,}

\newcommand{\pt}{\partial}

\DeclareMathOperator*{\esup}{\text{\rm ess\,sup}}

\newtheorem{proposition}{Proposition}[section]
\newtheorem{theorem}{Theorem}[section]
\newtheorem{lemma}{Lemma}[section]
\newtheorem{corollary}[theorem]{Corollary}

\theoremstyle{definition}
\newtheorem{remark}{Remark}[section]

\makeatletter

\@addtoreset{equation}{section}
\makeatother

\usepackage{layout}
\usepackage{xcolor}
\title{Forced rapidly dissipative Navier--Stokes flows}
\author{
Lorenzo Brandolese \\
{\normalsize Institut Camille Jordan} \\
{\normalsize Universit\'{e} Lyon 1 }\\
{\normalsize E-mail:~\texttt{brandolese@math.univ-lyon1.fr}}
\and
Takahiro Okabe \\
{\normalsize Graduate School of Engineering Science}\\
{\normalsize Osaka University}\protect\\
{\normalsize E-mail:~\texttt{okabe@sigmath.es.osaka-u.ac.jp}}
}
\newcommand{\dd}{{\rm\,d}}
\newcommand{\w}{\mathcal{W}}

\date{}
\begin{document}
\maketitle
%%%%%%%%%%%%%%%%%%%%%%%%%%%%%%%%%%%%%%%%%%%%%%%%%%
\begin{abstract}
We show that,
by acting on a finite number of parameters of a compactly supported
control force, we can increase the energy dissipation rate of any small solution of the Navier--Stokes equations in $\R^n$.
The magnitude of the control force is bounded by a negative Sobolev
norm of the initial velocity. Its support can be chosen to be contained in an arbitrarily small region, in time or in space.
\end{abstract}

\textbf{Key word:} Navier-Stokes equations, Energy decay, Asymptotic profiles, Control force.
\\
\textbf{MSC(2010):} 35Q30; 76D05
%%%%%%%%%%%%%%%%%%%%%%%%%%%%%%%%%%%%%%%%%%%%%%%%%%
\section{Introduction}
Let $n\geq 2$. We consider the incompressible Navier-Stokes equations in $\R^n$:
\begin{equation*}\tag{NS}
\left\{
\begin{aligned}
&\pt_t u -\Delta u + u\cdot \nabla u 
+ \nabla \pi=\nabla \cdot {f}
\qquad \text{in } \R^n\times (0,\infty),\\
& \Div u=0
\qquad \text{in } \R^n\times (0,\infty),\\
&u(\cdot,0)=a 
\qquad \text{in } \R^n,
\end{aligned}\right.
\end{equation*}
Here $a=(a_1,\ldots, a_n)$ and $u=(u_1,\ldots,u_n)(\cdot,t)$ are the velocity field at the time $0$ and $t$ respectively, the scalar field $\pi$ is the pressure and $f=f(x,t)$ is a $n\times n$ matrix-valued function.

The main issue of this paper is the large-time energy dissipation of solutions. An extensive overview of
this topic is contained in the survey paper~\cite{BraS}. See also, e.g., \cite{HLZZ, Vila, XuZ} for a  small sample of the most recent contributions.

A well known consequence of classical Wiegner's theorem~\cite{Wiegner} is that
solutions of the unforced ($f\equiv0$) Navier--Stokes equations in $\R^n$ arising from well localized data dissipate their energy
at the rate $\|u(t)\|_2^2=\mathcal{O}(t^{-(n+2)/2})$ as $t\to+\infty$.
By definition, \emph{rapidly dissipative solutions} to (NS) are global solutions
such that the energy dissipates faster, i.e.,
such that
\[
    \|u(t)\|_2^2=o\bigl(t^{-(n+2)/2}\bigr)
    \qquad\text{as $t\to+\infty$}.
\] 
Contrary to Navier--Stokes flows with the usual energy dissipation rate 
$\mathcal{O}(t^{-(n+2)/2})$, such rapidly dissipative solutions are not easily constructed, even though their existence is known \cite{BraS, GALW, Okabe Tsutsui JDE}: only few examples are available, which are all obtained putting stringent symmetry conditions. See, e.g.,  \cite{Brandolese MA 2004}. 
But the drawback of putting symmetry constraints is that the needed symmetries are unstable under small perturbation, so physically realistic flow will not fulfil them spontaneously,
in general.

The situation is different if one allows a control force.
In \cite{BraOka}, a new idea for an algorithmic construction of rapidly dissipative flows was proposed: for any small enough initial data  $a\in L^n(\R^n)$
it was shown how to construct an external force of divergence form, such that the corresponding solution is rapidly dissipative.
But the approach in \cite{BraOka} provides no quantitative information  on the control force. Moreover, the argument therein was quite involved, based on
a cumbersome hierarchy of scale-invariant decay estimates. 
In the present paper, developping further the ideas of \cite{BraOka}, we present a much more transparent construction. More importantly, our new approach allow us to quantify the size of the needed control force in terms of the size of the initial velocity, and it 
makes evidence that the rapid dissipation effect of the solution can be always achieved by \emph{acting on a finite number of parameters of an arbitrarily given control force}.

In particular, our result shows that the forcing term $f=(f_{k\ell})$
can always chosen of the form $f=A\chi$, where $\chi$ is, e.g.,
a rescaled arbitrary cut-off function (or an indicator function), 
satisfying $\int_0^\infty\!\!\int_{\R^n}\chi=1$ and $A=(A_{k\ell})$ is a $n\times n$ constant symmetric matrix 
(depending on the initial data)
such that $|A_{k\ell}|\le \|a\|_{\dot H^{-1}}^2$.
In particular, we can force any Navier--Stokes flow, with initially small velocity, to be rapidly dissipative by just tuning, at most, $2+n(n+1)/2$ parameters of an essentially arbitrary external forcing. If the function $\chi$ is chosen to be 
compactly supported in space-time, then the control force just acts on a bounded region and during a finite time interval.
In addition, we can achieve the result by either choosing to act with a control force on an arbitrarily fixed small spatial region (for a possibly long, but finite, time) or for an arbitrarily short time interval (in a possibly wide, but finite, spatial region), the size of $f$ only depending on the $\dot H^{-1}$-norm of the initial data.

From the technical point of view, a novelty of the present paper
is that we do not directly rely on asymptotic profiles established in earlier works \cite{Carpio, Fujigaki Miyakawa SIAM, GALW}, as those would require to put too stringent conditions on the data. Our approach is based on a new simple abstract result, that can be used to describe the large time asymptotics for a large class of integral equations.

\section{Statement of the main result}
\label{sec:statement}

Let us introduce, for $1\le r\le \infty$ the following function spaces:
\[
\begin{split}
&X_r=\bigl\{v\in L^\infty_{\rm loc}\bigl(0,\infty;L^r(\R^n)\bigr)
\colon \|v\|_{X_r}=
\esup\limits_{t>0}\,t^{\frac12-\frac{n}{2r}}\|v(t)\|_r<\infty\bigr\},\\
&Y_r=\bigl\{f\in L^\infty_{\rm loc}\bigl(0,\infty;L^r(\R^n)\bigr)
\colon \|f\|_{Y_r}=
\esup\limits_{t>0}\,t^{1-\frac{n}{2r}}\|f(t)\|_r<\infty\bigr\}.\\
\end{split}
\]
We will not distinguish between scalar, vector or tensor-valued fields in our notations, as this does not create confusion and allows us to simplify the notations.

Notice that $X_r$ is usual Kato's space for the velocity field
and such space is left invariant by the natural scaling of of (NS), $u\mapsto u_\lambda$,
with $u_\lambda(x,t)=\lambda u(\lambda x, \lambda^2 t)$.
See \cite[Chap.5]{BCD}.
The space $Y_r$ is the corresponding natural space for the forcing term~$f$, and it is left invariant by the natural scaling,
i.e., $\|f_\lambda\|_{Y_r}=\|f\|_{Y_r}$, where 
$f_\lambda(x,t)=\lambda^2f(\lambda x,\lambda^2 t)$.

We assume the reader to be familiar with Besov spaces.
We will mainly use homogeneous Besov of the form $\dot B^{s}_{r,q}(\R^n)$, with $1\le r,q\le\infty$ and negative regularity~$s$.
These are normed by 
$f\mapsto \bigl\| 2^{js}\|\Delta_j f\|_r\bigr\|_{\ell^q(\Z)}$,
where $(\Delta_j f)_{j\in\Z}$ is the Littlewood--Paley decomposition
of the tempered distribution~$f\in\mathcal{S}'_h$.  (See~\cite[Chap.~2]{BCD}).

We define $\dot B^{s}_{r,c_0}(\R^n)$ as the closed subspace of 
$\dot B^{s}_{r,\infty}(\R^n)$ such that
\[
\lim_{j\to-\infty} 2^{js}\|\Delta_j f\|_r=0.
\]
For any $1\le q<\infty$, we have of course the inclusions 
$\dot B^{s}_{r,q}(\R^n)\subset \dot B^{s}_{r,c_0}(\R^n)
\subset \dot B^{s}_{r,\infty}(\R^n)$.

The following fact on the Navier--Stokes equation is well known:
Let $2\le n<r<\infty$ and
$
a\in 
\dot B^{-1+n/r}_{r,\infty}(\R^n)$, with 
$\nabla\cdot a=0$ and $f\in Y_r$.
Then there exists a constant $\eta>0$ (only dependent on~$n$ and~$r$),
such that
if
\begin{equation*}
 \tag{A}
 \|a\|_{\dot B^{-1+n/r}_{r,\infty}}+\|f\|_{Y_r}<\eta,
\end{equation*}
then there exists a solution $u\in X_r$ to (NS), such that 
$\|u\|_{X_r}<2\bigl(\|a\|_{\dot B^{-1+n/r}_{r,\infty}}+\|f\|_{Y_r}\bigr)$, which
is uniquely defined in the
ball $\{v\in X_r\colon \|v\|_{X_r}<2\eta\}$.
This is a straightforward generalization of Cannone's  analysis of Kato's construction of mild solutions in $X_r$ when an external forcing is present.
See, e.g., \cite[Chapt. 5]{BCD}. See also \cite{Can, Lem, Lem21}.

Since the work of Miyakawa and Schonbek \cite{Miyakawa Schonbek}, we know that
the large time behavior of the velocity field at fast decay rates --- when the large time asymptotics is no longer governed by the linear part --- is related to
the algebraic properties of the energy matrix of the flow
\begin{equation}
\label{EMi}
\int_0^\infty\!\!\!\int_{\R^n} (u\otimes u)(y,s)\dd y\dd s.
\end{equation}
The starting point of our analysis will be the construction of solutions such that
the energy matrix is well defined, i.e., solutions such that $u\in L^2(\R^n\times \R^+)$. 
This will be done in Proposition~\ref{prop:exi}, by adding to the above smallness condition~(A)
the condition~$a\in \dot H^{-1}(\R^n)$. This additional condition 
is very natural, because the latter
is necessary and sufficient in order to have that $e^{t\Delta}a\in L^2(\R^n\times\R^+)$.

We now state our main result. 
Next theorem, roughly, asserts the following: if $a$ is an essentially \emph{arbitrary small and well localized initial data}
and if $\chi\in L^\infty_c(\R^n\times \R^+)$ is \emph{an arbitrary forcing profile with non-zero integral} then, 
after modifying the scalar function $\chi$ by multiplying it by a suitable $n\times n$ matrix $A$ with constant coefficients,
the unique global solution of the Navier--Stokes equation with external force 
$\nabla\cdot (A\chi)$ is rapidly dissipative.
  
\begin{theorem}
\label{th:theo1} 
Let $2\le n<r<\infty$.
Let
\[
a\in 
\dot B^{-1+n/r}_{r,\infty}(\R^n)\cap \dot H^{-1}(\R^n),
\qquad \nabla\cdot a=0.
\]
Let  $\chi\in L^\infty_c(\R^n\times\R^+)$, such that 
$\int_0^\infty\!\!\int\chi=1$.
There exist $\eta_0>0$ (only dependent on~$n$ and~$r$), 
and a constant real matrix $(\sigma_{k\ell})$ 
(dependent on $n$, $r$ and also on $a$, $\chi$), with
\[
|\sigma_{k\ell}|\le 1
\qquad(k,\ell=1,\ldots,n),
\]
such that if 
\begin{equation}
\label{small:all}
\begin{cases}
\|a\|_{\dot B^{-1+n/r}_{r,\infty}}+ \|a\|_{\dot H^{-1}}^2\|\chi\|_{Y_r}<\eta_0\\
 \|a\|_{\dot H^{-1}}
 \|\chi\|_{L^{\frac{4+2n}{4+n}}(\R^n\times \R^+)} <\eta_0
\end{cases}
\end{equation}
and if
\[
f=(f_{k\ell}),\quad 
\text{with}\quad
f_{k\ell}(x,t)=\sigma_{k\ell}\,\|a\|_{\dot H^{-1}}^2\chi(x,t), 
\]
then there exists a global solution $u\in X_r\cap L^2(\R^n\times\R^+)$~to (NS)
such that
\begin{equation}
\label{eq:ct}
\lim_{t\to+\infty} t^{\frac{1}{2}+\frac{n}{2}(1-\frac{1}{q})}\|u(t)-e^{t\Delta}a\|_q=0
\qquad
\textstyle
\text{for all $1\le q\le \infty$}.
\end{equation}
The above solution $u$ is rapidly dissipative, i.e.,
$\|u(t)\|_2^2=o\bigl(t^{-(n+2)/2}\bigr)$ as $t\to+\infty$,
if and only if the initial data~$a$ belongs also to $\dot B^{-(n+2)/2}_{2,c_0}(\R^n)$.
\end{theorem}

Next corollary  puts in evidence that we can always force the rapid dissipation of the solution by acting with a force on a tiny compact region in space-time,
the actual size of such a region depending on the size of the $\dot H^{-1}$-norm of the initial velocity. 
In addition, we can choose $\nabla\cdot f$ to be a distribution supported in a set of Lebesgue measure zero in space-time (by taking $\phi$ and $\psi$ to be indicator functions in the corollary below).

\begin{corollary}
\label{cor:uni}
Let  $a\in \dot B^{-1+n/r}_{r,\infty}(\R^n)\cap \dot H^{-1}(\R^n)$, with $n<r<\infty$.
Let $R,R'>0$ and 
\[
\chi(x,t)=R^nR'\phi(Rx)\,\psi(R't),
\]
where $\phi$ is a $L^\infty(\R^n)$-function supported in the unit cube $[0,1]^n$ and $\psi$ is a $L^\infty(\R^+)$-function supported in the interval~$[0,1]$, both with integral equal to one.
There exists $\eta_0'>0$ (depending on $\phi$, $\psi$, $n$ and $r$), such that if
\begin{equation}
\label{sysac}
\begin{cases}
\|a\|_{\dot B^{-1+n/r}_{r,\infty}}<\eta_0'\\
\|a\|_{\dot H^{-1}}^2 R^{n(1-\frac{1}{r})}(R')^{\frac{n}{2r}}<\eta_0'\\
\|a\|_{\dot H^{-1}} R^{\frac{n^2}{4+2n}} (R')^{\frac{n}{4+2n}}<\eta_0',
\end{cases}
\end{equation}
then the conclusion of Theorem~\ref{th:theo1} applies.
\end{corollary}

It is important to realize that, in order to fulfill the size conditions~\eqref{sysac}, at least one of $R$ and $R'$ can be taken arbitrary large, no matter how large is $\|a\|_{\dot H^{-1}}$.
Thus, we can force the rapid dissipation effect by acting with a force on a bounded region in space-time, either during an arbitrarily small time interval, or otherwise during a arbitrarily small spatial region.

Theorem~\ref{th:theo1} and its corollary improve our previous result~\cite{BraOka} in several ways.
A first important improvement is the precise quantification of the magnitude and of the size of the support for the admissible control forces.
Moreover, the functional setting is more general and the required smallness condition on $a$ on the Besov norm $\dot B^{-1+n/r}_{r,\infty}(\R^n)$ 
is weaker the smallness condition on $\|a\|_n$.
Furthermore, the additional condition $a\in \dot B^{-(n+2)/2}_{2,c_0}(\R^n)$ for the last conclusion of the theorem is optimal, and much weaker
than the corresponding condition $a\in L^1(\R^n,(1+|x|)\dd x)$, with vanishing moments up to the first order, that we put in~\cite{BraOka}.
Lastly, in the present paper, the proof is more transparent and self-contained: indeed, rather than relying on Fujigaki and Miyakawa's asymptotic profiles, we deduce~\eqref{eq:ct} from an abstract result (Theorem~\ref{th:asy} below) that should have an independent interest.

\paragraph{Overview of the proof of Theorem~\ref{th:theo1}.\!\!\!}
The integral formulation of the Navier-Stokes equations is used in Section~\ref{sec:L2L2}, to establish sufficient conditions on $a$ and $f$,
in order to ensure the existence and the uniqueness of a solution to (NS), belonging to $u\in L^2(\R^n\times\R^+)$, so that the energy matrix~\eqref{EMi} of the flow is well defined.

Section~\ref{sec:revisit} is self-contained. It is devoted to the proof
of Theorem~\ref{th:asy}. Given a kernel $M$ with appropriate scaling properties and a function 
$\mathcal{W}\in L^1(\R^n\times \R^+)$, Theorem~\ref{th:asy}
provides the asymptotic profiles as $t\to+\infty$, in $L^q(\R^n)$-norms, of space-time integrals of the form
\[
(x,t)\mapsto \int_0^\infty\!\!\!\int_{\R^n} M(x-y,t-s)\mathcal{W}(y,s)\dd y\dd s.
\]
Section~\ref{sec:hea} is also self-contained and studies the decay properties of
the heat equation in terms of Besov spaces. It characterizes the Besov-type
space $\dot B^{-s}_{r,c_0}(\R^n)$ in terms of time decay for solutions of the heat equation.

In Section~\ref{sec:algo}, 
the control forcing term~$f$ is constructed as the limit of an inductively defined sequence of control forces. 
In this sequence, the velocity
field for the Navier-Stokes equations at step $m$
is obtained by solving the Navier--Stokes equations using the control force at the previous step.
At each step of the algorithm, the components of the control forcing term are suitable constant
multiples of a fixed function $\chi$: such constants at the step $m+1$ are defined in terms of the energy matrix of the solution at step~$m$.
Estimates
are produced to show that the sequence so constructed converges to a
solution to the Navier-Stokes equations.

The proof of Theorem~\ref{th:theo1} is given in Section~\ref{sec:conclusion}, and relies on the results of Sections~\ref{sec:L2L2}--\ref{sec:algo}. 
In particular, Theorem~\ref{th:asy} is applied with
$\mathcal{W} = u\otimes u-f$ and $M$ the kernel of $e^{t\Delta}\PP\text{div}$, where $\PP$ is the
Leray projector. 
The result of Section~\ref{sec:hea} is useful for the last assertion of~Theorem~\ref{th:theo1}.

\medskip
From now, the integral symbol $\int$ without limits will denote the integral in the whole space $\int_{\R^n}$.

\section{Constructing solutions in $L^2(\R^n\times \R^+)$}
\label{sec:L2L2}
Let us introduce the linear functional $\Phi$ defined by
\begin{equation}
\label{eq:Phi}
\Phi(f)(t)=\int_0^t e^{(t-s)\Delta}\PP\nabla\cdot f(s)\dd s.
\end{equation}
We also introduce also the bilinear operator
\begin{equation}
\label{Gop}
 G(u,v)(t)
=-\int_0^t e^{(t-s)\Delta}\PP \nabla\cdot(u\otimes v)(s)\text{\,d}s.
\end{equation}

Then the integral formulation of (NS) reads
\begin{equation}
 \label{NS-ab}
 u(t)=e^{t\Delta}a+\Phi(f)(t)+G(u,u)(t),
 \qquad\nabla\cdot a=0.
\end{equation}

Let us now introduce two spaces (the former for the velocity field and the latter for the forcing term): these are less commonly used in the Navier--Stokes 
literature than the usual spaces $X_r$ and $Y_r$, because they obey to a different scaling, that is not the natural one.
But these spaces will be very useful in establishing large time decay estimates.
For $1\le p\le \infty$, let
\[
    \begin{split}
    &\overline{X}_p
    =
    \bigl\{v\in L^\infty_{\rm loc}\bigl(\R^+;L^p(\R^n)\bigr)\colon 
        \|v\|_{\overline{X}_p}
        =\esup_{t>0}t^{\frac{1}{2}+\frac{n}{2}(\frac{1}{2}-\frac{1}{p})}
        \|v(t)\|_p<\infty
    \bigr\},\\
    &
    \overline{Y}_p
    =
    \bigl\{ f\in L^\infty_{\rm loc}\bigl(\R^+;L^{p}(\R^n)\bigr) \colon 
        \|f\|_{\overline Y_p}
        =\esup\limits_{t>0} t^{1+\frac{n}{2}(\frac{1}{2}-\frac{1}{p})}
        \|f(t)\|_p<\infty
    \bigr\}.
    \end{split}
\]
The space $\overline Y_q$ will be used only for $2n/(n+2)\le q\le\infty$.
\begin{lemma}
\label{lem:heat}
Let $p,r\in[1,\infty]$.
The heat semigroup operator $a\stackrel{S}{\mapsto} (t\mapsto e^{t\Delta}a)$ is a 
isomorphism from the following linear spaces into their respective ranges $R(S)$:
\begin{itemize}
\item[i)]
$S\colon \dot B^{-1+n/r}_{r,\infty}(\R^n)\to X_{r}$ \,for $n<r<\infty$,
\item[ii)]
$S\colon \dot H^{-1}(\R^n)\to L^2(\R^n\times \R^+)$ and 
\item[iii)]
$S\colon \dot B^{-1-n(\frac12-\frac1p)}_{p,\infty}\to \overline X_p$ \,for\, $0\le \frac1p<\frac 12+\frac1n$.
\end{itemize}
\end{lemma}

\begin{proof}

The  function $(x,t)\mapsto e^{t\Delta}a(x)$ belongs to $X_{r}$ if and only if $a$ belongs to $\dot B^{-1+n/r}_{r,\infty}(\R^n)$
by the well known Besov space characterisation via the heat kernel. See \cite[Theorem 2.34]{BCD}
where the norm equivalence $\|e^{t\Delta}a\|_{X_{r}}\approx\|a\|_{\dot B^{-1+n/r}_{r,\infty}}$ is 
also established. Thus Item~i) is well known.
In the same way, Item~ii) relies on the identification $\dot{H}^{-1}(\R^n)=\dot B^{-1}_{2,2}(\R^n)$ and the
heat kernel characterization of the latter space. But a direct proof is also possible:
indeed,
\[
    \int_0^\infty \|e^{t\Delta}a\|_2^2\dd t =
    \int_0^\infty\!\!\!\int e^{-2t|\xi|^2}|\widehat a(\xi)|^2\dd\xi\dd t
    ={\textstyle\frac{1}{2}}\int |\xi|^{-2}|\widehat a(\xi)|^2\dd\xi
\]
and the second assertion follows from the definition of $\dot H^{-1}(\R^n)$.
In particular, the function $(x,t)\mapsto e^{t\Delta}a(x)$ belongs to $L^2(\R^n\times \R^+)$ if and only if $a\in \dot H^{-1}(\R^n)$, and in this case
\[
    \int_0^\infty \|e^{t\Delta}a\|_2^2\dd t 
    ={\textstyle\frac{1}{2}} \|a\|_{\dot H^{-1}}^2.
\]
 Notice that \cite[Theorem 2.34]{BCD} can be used also to establish
the norm equivalence  $\|e^{t\Delta}a\|_{\overline X_p}\approx\|a\|_{\dot B^{-1-n(\frac12-\frac1p)}_{p,\infty}}$,
that proves Item iii). 

\end{proof}

\begin{lemma}
\label{lem:L2f}
Let $p,r,q,q_1,q_2\in [1,\infty]$.
The linear functional $\Phi$ defined in~\eqref{eq:Phi} is continuous 
\begin{itemize}

\item[i)]
$\Phi\colon Y_q\to X_{r}$,\, for 
$0\le \frac{1}{r}\le \frac{1}{q}<\frac{1}{r}+\frac{1}{n}$, with $q<\infty$.
\item[ii)] 
$\Phi\colon L^2(\R^+;\dot H^{-1}(\R^n))\to L^2(\R^n\times\R^+)$,
\item[iii)]
$\Phi\colon \overline Y_q\to \overline X_p$, \quad with 
$\frac 1p\le\frac1q<\frac1p+\frac1n$, \quad 
$\frac1q>\frac12$.
\item[(iv)] 
More in general, for $1\le q_1,q_2,p\le \infty$,
$\Phi\colon \overline Y_{q_1}\cap \overline Y_{q_2}\to \overline X_p$, \quad with $\frac 1p\le\frac1{q_1}<\frac1p+\frac1n$, \quad 
$\frac1{q_2}>\frac12$,\quad $\frac1{q_2}\ge \frac1p$.
In particular, if $0\le \frac1p\le \frac1n+\frac12$, then 
$\Phi\colon\overline Y_{2n/(n+2)}\cap\overline Y_p\to\overline X_p$.
\end{itemize}
\end{lemma}

\begin{proof}
Let us recall that the kernel $F(\cdot,t)$ of $e^{t\Delta}\PP{\rm div}$ enjoys the scaling
properties
\[ F(\cdot,t)=t^{-\frac{n+1}{2}}F(\cdot/\sqrt t,1), \]
with $F(\cdot,1)\in L^1\cap L^\infty(\R^n)$. 
Moreover, $F(\cdot,1)$ is smooth and all its derivatives are bounded in $\R^n$ 
and decay to~$0$ at the spatial infinity. See, e.g., \cite[Proposition~11.1]{Lem}.
Therefore,
\begin{equation}
    \label{Lano}
    \|F(\cdot,t)\|_\alpha
    =\|F(\cdot,1)\|_\alpha \, t^{-\frac{n+1}{2}+\frac{n}{2\alpha}},
    \qquad
    \text{for all $1\le \alpha\le \infty$}.
\end{equation}

For the first assertion, let us observe that,
\[
    \Bigl\| \int_0^t e^{(t-s)\Delta}\PP\nabla\cdot f(s)\dd s\Bigr\|_{r}
    \lesssim 
    \int_0^t (t-s)^{-\frac{1}{2}-\frac{n}{2}(\frac{1}{q}-\frac{1}{r})}\|f(s)\|_q\dd s, 
\]
 where we applied~\eqref{Lano} 
 with $1+\frac{1}{r}=\frac{1}{\alpha}+\frac{1}{q}$. 
 From this, we deduce that there exists a constants  $\kappa$, only dependent on the space dimension $n$ and $r,q$,
such that, for any $f\in Y_n$,
\[
\|\Phi(f)\|_{X_r}\le \kappa\|f\|_{Y_q}.
 \]
 
Let us prove our second assertion. 
Indeed, for a.e. $\xi\in\R^n$,
\[
|\widehat{\Phi(f)}(\xi,t)| \lesssim \int_0^t e^{-(t-s)|\xi|^2}|\xi|\,
|\widehat f(\xi,s)|\dd s.
\]
For a.e. $\xi\in \R^n$, the function of the $t$-variable on the right-hand side is the convolution product in $\R$ of the 
function $t\mapsto |\xi|e^{t|\xi|^2}{\bf 1}_{\R^+}(t)$, with the function
$t\mapsto |\widehat f(\xi,t)|{\bf 1}_{\R^+}(t)$.
By $L^1$-$L^2$ Young inequality in one-dimension,
\[
\begin{split}
\int_0^\infty |\widehat{\Phi(f)}(\xi,t)|^2\dd t 
&\lesssim
\Bigl(\int_0^\infty|\xi| e^{-t|\xi|^2}\dd t\Bigr)^2
\Bigl( 
\int_0^\infty |\widehat f(\xi,t)|^2\dd t
\Bigr)\\
&=\int_0^\infty |\xi|^{-2}|\widehat f(\xi,t)|^2\dd t.
\end{split}
\] 
The second assertion now follows integrating over $\xi\in \R^n$ and
applying Parseval identity:
in particular, there exists a constants  $\kappa_n$, only dependent on the space dimension $n$, such that, 
for any $f\in L^2\bigl(\R^+;\dot H^{-1}(\R^n)\bigr)$,
\[
\int_0^\infty \|\Phi(t)\|_2^2\dd t\le \kappa_n \int_0^\infty \|f(t)\|_{\dot H^{-1}}^2\dd t.
\]

The third assertion is proved applying~\eqref{Lano} 
with $1+\frac{1}{p}=\frac{1}{\alpha}+\frac{1}{q}$ and $p\ge q$:
\[
    \begin{split}
    \|\Phi(f)(\cdot,t)\|_p
    &\lesssim 
    %\int_0^t \|F(t-s)\|_{\alpha}\|f(s)\|_{q}\dd s, \qquad \textstyle1+\frac1p=\frac1\alpha+\frac1q, \quad p\ge q\\
    %&\lesssim
    \int_0^t (t-s)^{-\frac{1}{2}-\frac{n}{2}(\frac{1}{q}-\frac{1}{p})}
        s^{-1-\frac{n}{2}(\frac{1}{2}-\frac{1}{p})} \dd s\,\|f\|_{\overline Y_q}\\
    &\lesssim
    t^{-\frac{1}{2}-\frac{n}{2}(\frac{1}{2}-\frac{1}{p})}\,\|f\|_{\overline Y_q}.
\end{split}
\] 
The conditions on $p,q$ insure that the above integral converges.

Let us prove our last assertion.
If $f\in \overline Y_{q_1}\cap \overline Y_{q_2}$, then we can split
the intergral at $t/2$ before applying twice Young's convolution estimate
and we get
\[
    \begin{split}
    \|\Phi(f)(\cdot,t)\|_p
    &\lesssim
    t^{-\frac{1}{2}-\frac{n}{2}(\frac{1}{q_2}-\frac{1}{p})} 
    \int_0^{t/2} s^{-1-\frac{n}{2}(\frac{1}{2}-\frac{1}{q_2})}
    \dd s\,\|f\|_{\overline Y_{q_2}}\\
    &\qquad\qquad\qquad
    +t^{-1-\frac{n}{2}(\frac{1}{2}-\frac{1}{q_1})}
    \int_{t/2}^t (t-s)^{-\frac{1}{2}-\frac{n}{2}(\frac{1}{q_1}-\frac{1}{p})}
    \dd s \|f\|_{\overline Y_{q_1}}\\
    &\lesssim
    t^{-\frac{1}{2}-\frac{n}{2}(\frac{1}{2}-\frac{1}{p})}
    \,\bigl(\|f\|_{\overline Y_{q_1}}+\|f\|_{\overline Y_{q_2}}\bigr)
\end{split}
\]
and the two integrals converge under the conditions we put on $p$, $q_1$ and $q_2$. 
\end{proof}

\begin{remark}
\label{rem:vari}
There are many possible variants to conclusion ii).
For example, the linear operator $\Phi$ is also continuous as a map 
\[
    \Phi\colon L^{\alpha}\bigl(\R^+;L^\beta(\R^n)\bigr)\to L^2(\R^n\times\R^+),
    \qquad 1<\alpha< 2,\quad 1\le \beta\le 2,\quad
    \frac{1}{\alpha}+\frac{n}{2\beta}=1+\frac{n}{4}.
\]
In particular, choosing $\alpha=\beta$.
\[
    \Phi\colon L^{\frac{4+2n}{4+n}}(\R^n\times\R^+)\to L^2(\R^n\times\R^+),
\]
Indeed,
\[
    \|\Phi(f)(t)\|_{L^2(\R^n)}
    \lesssim 
    \int_0^t (t-s)^{-\frac{1}{2}-\frac{n}{2}(\frac{1}{\beta}-\frac{1}{2})}
    \|f(s)\|_\beta\dd s.
\]
Now, the map $t\mapsto t^{-1/2+(n/2)(1/2-1/\beta)}{\bf 1}_{\R^+}$ belongs to the 
Lorentz space
$L^{\delta,\infty}(\R)$, with $\frac1\delta=\frac12-\frac n2(\frac12-\frac1\beta)$.
But, by the relation between $\alpha$ and $\beta$, we have $1+\frac12=\frac1\delta+\frac1\alpha$. 
When $f\in L^{\alpha}\bigl(\R^+;L^\beta(\R^n)\bigr)$, the map 
$t\mapsto \|f(t)\|_\beta{\bf 1}_{\R^+}$ belongs to $L^\alpha(\R)\subset L^{\alpha,2}(\R)$ and so 
$t\mapsto \|\Phi(f)(t)\|_2$
belongs to $L^2(\R^+)$ by the Young convolution inequality in Lorentz spaces (see \cite[Proposition~2.4]{Lem}).
This will be especially useful when $n=2$: in the two-dimensional case
the condition $f\in L^2\bigl(\R^+;\dot H^{-1}(\R^2)\bigr)$ is somehow too stringent,
but such a condition can conveniently be replaced, e.g., by the condition $f\in L^{4/3}(\R^2\times\R^+)$.

\end{remark}

\begin{lemma}
\label{lem:mix}
The 
bilinear operator is continuous:
\begin{itemize}
\item[i)]
$G\colon X_r\times X_r\to X_r$, \quad $n<r<\infty$.
\item[ii)]
$
G\colon L^2\bigl(\R^+;L^2(\R^n)\bigr)\times X_r\to L^2\bigl(\R^+;L^2(\R^n)\bigr)
$, \, for $2\le n<r<\infty$, and
\item[iii)]
$G\colon (\overline X_2\cap \overline X_p)\times X_{r}\to (\overline X_2\cap \overline X_p)$,
\quad for $0\le \frac1p\le(\frac12+\frac1r)\wedge (1-\frac1r)$
and $0<\frac1r<\frac1n$.
\end{itemize}
\end{lemma}

\begin{proof}
The first assertion follows from the estimate
\begin{equation}
 \label{kato-est}
 \|G(u,v)\|_{X_{r}}\lesssim \|u\|_{X_{r}}\|v\|_{X_{r}}.
\end{equation}
This is well known. In fact, the above bilinear estimate in Kato's space
is just one particular instance of the more general boundedness property  of the bilinear operator in Kato's space
 $G\colon X_p\times X_q\to X_r$, with $0<\frac 1p+\frac1q\le 1$ and
$\frac1r\le \frac1p+\frac1q<\frac1r+\frac1n$.
(See~\cite[Lemma 5.29]{BCD} at least for $n=3$, but the method goes through for $n\ge2$).

Let us prove the second assertion.
Let $r'$ be the conjugate exponent of $r$. Notice that $r'<2<r$ and 
 that $\frac{1}{2}+1=\frac{1}{r'}+(\frac12+\frac1r)$.
Combining the H\"older and the Young inequality we find
\[
    \begin{split}
    \|G(u,v)(\cdot,t)\|_{2}
    &\le \int_0^t \|F(t-s)\|_{r'}\|u(s)\|_{2}\|v(s)\|_r\dd s\\
    &\lesssim \|v\|_{X_r}\int_0^t (t-s)^{-\frac12-\frac{n}{2r}}\|u(s)\|_2 \, s^{-\frac12+\frac{n}{2r}}\dd s.
    \end{split}
\]
But $s\mapsto \|u(s)\|_2$ belongs to $L^2(\R^+)$ and 
$s\mapsto s^{-\frac12+\frac{n}{2r}}$ belongs to $L^{2r/(r-n),\infty}(\R^+)$.
Therefore the function $s\mapsto \|u(s)\|_2s^{-\frac12+\frac{n}{2r}}$
belongs to $L^{(2r)/(2r-n),2}(\R^+)$ by the H\"older inequality in Lorentz spaces.
On the other hand, the function $s\mapsto s^{-\frac12-\frac{n}{2r}}$ belongs
to $L^{(2r)/(r+n),\infty}(\R^+)$.
By the Young convolution inequality in Lorentz space~\cite[Chap. 2]{Lem}, 
$L^{(2r)/(2r-n),2}*L^{(2r)/(r+n),\infty}\subset L^{2,2}=L^2$.
Thus, the integral on the right-hand side, as a function of~$t$, belongs to $L^2(\R^+)$.
It then follows that
\[
\|G(u,v)\|_{L^2(\R^n\times\R^+)}=
\Bigl(\int_0^\infty \|G(u,v)(\cdot,t)\|_2^2\dd t\Bigr)^{1/2}
\lesssim \|u\|_{L^2(\R^n\times \R^+)}\,\|v\|_{X_r}.
\]

The third assertion is proved as follows: let 
$\frac{1}{p}\le\frac{1}{2}+\frac{1}{r}$,\, 
$\frac{1}{r}+\frac{1}{p}\le1$, and $0<\frac{1}{r}<\frac{1}{n}$.
%$\textstyle 1+\frac1p=\frac1\alpha+\frac1r+\frac1p$ and $p\ge\alpha$. 
Then
\[
    \begin{split}
    \|G(u,v)(\cdot,t)\|_p
    &\le 
    \int_0^{t/2}(t-s)^{-\frac{1}{2}-\frac{n}{2}(\frac{1}{2}+\frac{1}{r}-\frac{1}{p})}
        \|u(s)\|_2\|v(s)\|_r\dd s
    +
    \int_{t/2}^t (t-s)^{-\frac{1}{2}-\frac{n}{2r}}
        \|u(s)\|_p\|v(s)\|_r\dd s \\
    &\lesssim
    t^{-\frac{1}{2}-\frac{n}{2}(\frac{1}{2}+\frac{1}{r}-\frac{1}{p})}
        \int_0^{t/2} s^{-1+\frac{n}{2r}} \dd s \, 
        \|u\|_{\overline X_2}\|v\|_{X_r}\\
    &\qquad\qquad
    +
    \int_{t/2}^t (t-s)^{-\frac{1}{2}-\frac{n}{2r}}
    s^{-1+\frac{n}{2r}-\frac{n}{2}(\frac{1}{2}-\frac{1}{p})}
    \dd s\,\|u\|_{\overline X_p}\|v\|_{X_r}\\
    &\lesssim
    t^{-\frac{1}{2}-\frac{n}{2}(\frac{1}{2}-\frac{1}{p})}
    \Bigl( \|u\|_{\overline X_2} \|v\|_{X_r}+ \|u\|_{\overline X_p}\|v\|_{X_r}\Bigr).
    \end{split}
\]
Notice that $p=2$ is allowed in the above calculations.
The last assertion of the Lemma follows.
\end{proof}

We can now state the main result of this section. (The first item is classical):

\begin{proposition}
 \label{prop:exi}
 Let $2\le n<r<\infty$. 
 There exists $\eta>0$ and a constant $C_0>0$ such that:
\begin{itemize}
\item[i)] if $a\in \dot B^{-1+n/r}_{r,\infty}(\R^n)$,  $f\in Y_r$ and
\begin{equation}
 \label{eq: lll}
 \tag{A}
 \|a\|_{\dot B^{-1+n/r}_{r,\infty}}+\|f\|_{Y_r}<\eta,
\end{equation}
then there exists a solution $u\in X_{r}\cap X_\infty$ of (NS), unique in the ball 
$\{v\in X_r\colon \|v\|_{X_r}<2\eta\}$, such that
 \begin{equation}
 \label{eq:boun}
 \begin{split}
    & \|u\|_{X_r}
        \le C_0 \bigl( \|a\|_{\dot B^{-1+n/r}_{r,\infty}}+\|f\|_{Y_r}\bigr),\\
    &\|u\|_{X_\infty}
        \le C_0  \bigl( \|a\|_{\dot B^{-1+n/r}_{r,\infty}}+\|f\|_{Y_r}\bigr).
    \end{split}
 \end{equation}
\item[ii)] If $a$ and $f$ satisfy the condition of the first item
and, moreover, $a\in \dot H^{-1}(\R^n)$, 
$f\in L^2\bigl(\R^+;\dot H^{-1}(\R^n)\bigr)$, then
\begin{equation}
\label{eq:bound-u}
    \|u\|_{L^2(\R^n\times \R^+)} 
    \le {\textstyle\frac{4}{5}}\,\|a\|_{\dot H^{-1}}
    +C_0\|f\|_{L^2(\R^+;\dot H^{-1})}.
\end{equation}
The condition $f\in  L^2(\R^+;\dot H^{-1}(\R^n))$ can be replaced by
the condition $f\in L^{(4+2n)/(4+n)}(\R^n\times\R^+)$ modifying estimate~\eqref{eq:bound-u} accordingly.
\item[iii)] 
Furthermore, for $2 \leq p \leq \infty$, there exists $\eta^\prime\in (0,\eta)$
such that if $a \in \dot{H}^{-1}(\R^n)\cap \dot{B}^{-1+n/r}_{r,\infty}(\R^n)$, 
 $f \in \overline{Y}_1\cap \overline{Y}_p$, $f\in Y_r$ and if
\begin{equation*}
    \|a\|_{\dot{B}^{-1+n/r}_{r,\infty}} + \|f\|_{{Y}_r} <\eta^\prime,
\end{equation*}
then additionally $u \in \overline{X}_p$ and 
\begin{equation}
\label{ezp}
    \|u\|_{\overline X_p}
    \le C \bigl( \|a\|_{\dot H^{-1}} 
        +\|f\|_{\overline Y_1} + \|f\|_{\overline Y_p}\bigr).
\end{equation}
\end{itemize}
\end{proposition}

\begin{proof}
The first statement is just the application of Item i) of our three previous Lemmas and the usual fixed point theorem \cite[Theorem 13.2]{Lem}, in the space $X_r$.
For this one needs to choose $\eta>0$ such that 
\[
4\eta\sup\|G(u,v)\|_{X_r}<1,
\]
where the supremum is taken over all $u,v\in X_r$ with $\|u\|_{X_r}=\|v\|_{X_r}=1$.
This supremum is finite because of~\eqref{kato-est}.
The solution $u$ is obtained as the limit, in the $X_r$-norm, of the sequence of approximate solutions
\begin{equation}
\label{iterkato}
u_{k+1}=u_0+\Phi(f)+G(u_k,u_k),\qquad
k\in\N,
\qquad
\text{with \; }
u_0(t)=e^{t\Delta}a.
\end{equation}
Of course, these approximate solutions are bounded, in the $X_r$-norm by the right-hand side of~\eqref{eq:boun}.
But, as observed right after~\eqref{kato-est}, we do have also the boundedness of
$G\colon X_\infty\times X_r\to X_\infty$. We combine this with the boundedness of the heat kernel operator
$S\colon \dot B^{-1+n/r}_{r,\infty}(\R^n)\to X_\infty$ (this is nothing but a restating of the well the continuous embeddding
$\dot B^{-1+n/r}_{r,\infty}(\R^n)\subset \dot B^{-1}_{\infty,\infty}(\R^n)$)
and the boundedness $\Phi\colon Y_r\to X_\infty$ (already established).
All together, this implies 
that (all the constants are independent on $k$):
\[
\begin{split}
\|u_{k+1}\|_{X_\infty}
&\le \|u_0\|_{X_\infty}+\|\Phi(f)\|_{X_\infty}
+\|G(u_k,u_k)\|_{X_\infty}\\
&\lesssim
\|a\|_{B^{-1+n/r}_{r,\infty}}+\|f\|_{Y_r}
+\|u_k\|_{X_r}\|u_k\|_{X_\infty}\\
&\lesssim
\|a\|_{B^{-1+n/r}_{r,\infty}}+\|f\|_{Y_r}+2\eta\|u_k\|_{X_\infty}.
\end{split}
\]
We can iterate this estimate. Replacing, if necessary, $\eta>0$ by a smaller constant (only depending on $n$ and $r$),
we get the second of~\eqref{eq:boun}.

Let us prove our second assertion.
Applying Item~ii) of our three previous Lemmas, we can estimate, for some $C_1>0$,
\[
    \begin{split}
    \|u_{k+1}\|_{L^2(\R^n\times \R^+)}
        &\le 
        \|u_0\|_{L^2(\R^n\times \R^+)}+\|\Phi(f)\|_{L^2(\R^n\times \R^+)}
        +\|G(u_k,u_k)\|_{L^2(\R^n\times \R^+)}\\
    &\leq
        {\textstyle\frac{1}{\sqrt{2}}}\|a\|_{\dot H^{-1}}+\kappa_n\|f\|_{L^2(\R^+;\dot H^{-1})}
        +C_1\|u_k\|_{X_r}\|u_k\|_{L^2(\R^n\times \R^+)}\\
    &\leq
        {\textstyle\frac{1}{\sqrt{2}}}\|a\|_{\dot H^{-1}}+\kappa_n \|f\|_{L^2(\R^+;\dot H^{-1})}
        +2\eta C_1\|u_k\|_{L^2(\R^n\times \R^+)}.
    \end{split}
\]
Once more, replacing, if necessary, $\eta>0$ by a smaller constant 
(only dependent on $n$ and $r$) noting that $1/\sqrt{2}=0.707\dots$, then iterating this estimate we deduce that
\begin{equation}
    \|u_k\|_{L^2(\R^n\times \R^+)}
    \leq  {\textstyle\frac{4}{5}} \|a\|_{\dot H^{-1}}
    +C_0 \|f\|_{L^2(\R^+;\dot H^{-1})}
\end{equation}
for all $k\in\N$. 
Passing to the limit as $k\to+\infty$ we deduce from this estimate~\eqref{eq:bound-u}.
Of course, when one assumes $f\in L^{(4+2n)/(4+n)}(\R^n\times \R^+)$, 
instead of $f\in L^2\bigl(\R^+;\dot H^{-1}(\R^n)\bigr)$, then Remark~\ref{rem:vari} should be used.
 
The proof of our third assertion relies on the same ideas. Let $2\le p\le \infty$.
We first make use of the continuous embedding
$\dot H^{-1}(\R^n)\subset \dot B^{-1-n(\frac12-\frac1p)}_{p,\infty}(\R^n)$
and of Item~{iii)} of Lemma~\ref{lem:heat}, to obtain that $\|e^{t\Delta}a\|_{\overline X_p}\lesssim\|a\|_{\dot H^{-1}}$.
Next, by Item~{iv)} of Lemma~\ref{lem:L2f} (with $q_1=p$ and $q_2=1$), we have $\|\Phi(f)\|_{\overline X_p}\lesssim \|f\|_{\overline Y_1}+\|f\|_{\overline Y_p}$. 
And by Item~{iii)} of Lemma~\ref{lem:mix},
$G\colon (\overline X_2\cap \overline X_p)\times X_r\to \overline X_p$.
Using the same approximation argument as before we first deduce that $u\in \overline X_2$,
next that $u\in \overline X_p$ and that estimate~\ref{ezp} holds.
\end{proof}

\section{Revisiting Fujigaki and Miyakawa asymptotic profiles}
\label{sec:revisit}

In this section we establish a self-contained result, on the asymptotic behavior of an integral term. Theorem~\ref{th:asy} below will be one of the ingredients for the proof of Theorem~\ref{th:theo1}, but it has an independent interest.

Let $M$ be a measurable function on $\R^n\times \R^+$ 
which satisfies the following scaling properties 
\begin{subequations}
\label{ass:M}
\begin{equation}
\label{scalingM}
    M(x,t)=t^{-\frac{n+1}{2}}M(x/\sqrt t,1), \qquad x\in\R^n,\;t>0.
\end{equation}
We also assume that
\begin{equation}
M(\cdot,1) \in L^1(\R^n)\cap L^\infty(\R^n).
\end{equation}
\end{subequations}

Next theorem is inspired by a result by Carpio \cite{Carpio} and  Fujigaki-Miyakawa~\cite{Fujigaki Miyakawa SIAM}. Later on, we will apply it with $\w= u\otimes u-f$ and $M=F$.
Its formulation is closer to~\cite[Theorem~2.1]{BRA-Lon}, but it differs in that the scaling relation of $M$ in \cite{BRA-Lon} are not the same and the decay rate, therein,  were slower. Moreover, we removed unessential regularity conditions on~$M$. 

\begin{theorem}
\label{th:asy}
%\begin{enumerate}
%\item
\label{item:part1}
Let $n\ge1$,  $\w\in L^1(\R^n\times\R^+)$, with\,  $\|\w(t)\|_1=\mathcal{O}(\frac1t)$ as $t\to+\infty$.
Let us introduce the constant 
$ \lambda=\int_0^\infty\!\!\int \w(y,s)\dd y\dd s$
and let also
\begin{equation*}
 \label{integralphi}
 \Phi(x,t)=\int_0^t\!\!\int M(x-y,t-s)\w(y,s)\dd y\dd s.
\end{equation*}
Then, as $t\to+\infty$,
\begin{equation}
    \label{pro:Phi}
        \biggl\| \Phi(t) - \lambda  M(\cdot,t)  \,\biggr\|_q
        = o\bigl(t^{-\frac{1}{2}-\frac n2(1-\frac1q)}\bigr), 
        \qquad \text{with}\quad
            1\le  q<\frac{n}{n-1}.
      \end{equation}
%In particular, when $\lambda\not=0$, there exist two constants $c_q,c_q'>0$, independent on~$\w$, such that, for~$t>\!\!\!>1$,
%\begin{equation}
%\label{bound:Phi}
% \lambda c_q \, t^{-\frac n2(1-\frac1q)-1/2} \le \|\Phi(t)\|_q \le \lambda c'_q\, t^{-\frac n2(1-\frac1q)-1/2}.
%\end{equation}
%\item
%\label{item:part2}
The above results extend to
$\frac{n}{n-1}\le q\le \infty$, provided $\w$ satisfies also 
$\|\w(t)\|_\beta =\mathcal{O}(t^{-1-\frac n2(1-\frac1\beta)})$ as $t\to\infty$, 
for some $\beta$ such that 
$\frac1q\le \frac{1}{\beta}<\frac1q+\frac1n$. 
(The conclusion in the case $q=\infty$ requires the additional assumption that $M(\cdot,1)\in C_0(\R^n)$). 
%\end{enumerate}
\end{theorem}

\begin{proof}
%
%First of all, observe that the assumptions we put on $M$ imply
%the existence of a constant $C>0$ such that,
%for all $1\le q\le\infty$, and $m=0,1$:
%\begin{equation}
%    \label{decay-q}
%    \| \partial_x^mM(\cdot,t)\|_q=C\,t^{-\frac{1+m}{2}-\frac n2(1-\frac1q)}, \qquad
%    \| \partial_t^mM(\cdot,t)\|_q=C\,t^{-\frac{1}{2}-m-\frac n2(1-\frac1q)}.
%   \end{equation}

Let $\eta>0$ arbitrary and $a_\eta>0$ to be chosen later, with $\frac12<a_\eta<1$.
We decompose
 \[
 \begin{split}
  \Phi(x,t)-&\lambda M(x,t)\\
  \qquad
  &=\int_0^t\!\!\int M(x-y,t-s)\w(y,s)\dd y\dd s-  M(x,t) \int_0^\infty\!\!\!\int \w(y,s)\dd y \dd s\\
  &=(I_1+I_2+I_3+I_4)(x,t),
 \end{split}
 \]
where
\[
\begin{split}
 I_1(x,t)&=-M(x,t)\int_{a_\eta t}^\infty\!\int \w(y,s)\dd y \dd s,\\
 I_2(x,t)&=\int_0^{a_\eta t}\!\!\int [M(x,t-s)-M(x,t)]\w(y,s)\dd y\dd s,\\
 I_3(x,t)&=\int_0^{a_\eta t}\!\!\int[M(x-y,t-s)-M(x,t-s)]\w(y,s)\dd y \dd s\qquad\text{and}\\
 I_4(x,t)&=\int_{a_\eta t}^t\!\!M(x-y,t-s)\w(y,s)\dd y\dd s.
 \end{split}
\]
Now,
\[
    \|I_1(t)\|_q
    \le 
    C_q\,t^{-\frac{1}{2}-\frac{n}{2}(1-\frac{1}{q})}\int_{a_\eta t}^\infty\!\int |\w(y,s)|\dd y \dd s
    =o\bigl(t^{-\frac{1}{2}-\frac{n}{2}(1-\frac{1}{q})}\bigr),
 \qquad \text{as $t\to+\infty$}.
\]
by the assumption $\w\in L^1(\R^n\times\R^+)$ and the dominated convergence theorem.

For the estimate of $I_2$, we start observing that
\[
\|I_2(t)\|_q
\le \int_0^{a_\eta t}\!\!\!\int  \|M(\cdot,t-s)-M(\cdot,t)\|_q|\w(y,s)|\dd y \dd s.
\]
By the scaling properties of~$M$,
\[
\begin{split}
\|I_2(t)\|_q
&\le 
t^{\frac{n}{2q}}\int_0^{a_\eta t}\!\!\!\int 
\textstyle
\bigl\|\bigl(t-s\bigr)^{-\frac{n+1}{2}} M\bigl(\frac{\sqrt{t}\,\cdot}{\sqrt{t-s\,}},1\bigr)-t^{-\frac{n+1}{2}}M(\cdot,1)\bigr\|_q \,|\w(y,s)|\dd y\dd s\\
&\le
t^{-\frac{n+1}{2}+\frac{n}{2q}}\int_0^\infty\!\!\!\int 
\textstyle
\bigl\|
\bigl(1-s/t\bigr)^{-\frac{n+1}{2}}
M\bigl(\frac{\cdot}{\sqrt{1-s/t\,}},1\bigr)-M(\cdot,1)\bigr\|_q \, |\w(y,s)|
\, {\bf 1}_{[0,a_\eta t]}(s)\dd y\dd s,
\end{split}
\]
where ${\bf 1}_S$ denotes the indicator function of the set~$S$.
For a.e. $s\ge0$ we have
\[
\bigl\|
\bigl(1-s/t\bigr)^{-\frac{n+1}{2}}\textstyle
M\bigl(\frac{\cdot}{\sqrt{1-s/t\,}},1\bigr)-M(\cdot,1)\bigr\|_q \to0,
\qquad\text{as $t\to+\infty$}.
\]
We used here that, if $\beta_n\to1$, then $g(\beta_n\cdot) \to g$ in $L^q(\R^n)$, when $g\in L^q(\R^n)$, which is proved by an elementary density argument for $1\le q<\infty$, but it is true also for $q=\infty$, provided $g\in C_0(\R^n)$.  

The integrand in the right-hand side above is dominated by
\[
\bigl((1-a_\eta)^{-\frac{n+1}{2}+\frac{n}{2q}}+1\bigr)\|M(\cdot,1)\|_q\,|\w(y,s)| \in L^1(\R^n\times\R^+).\]
By the dominated convergence theorem, we finally deduce
\[
    \|I_2(t)\|_q=o(t^{-\frac{1}{2}-\frac n2(1-\frac1q)})
 \qquad \text{as $t\to+\infty$}.
\]

In order to estimate the third integral  we make use once more of the scaling properties~\eqref{scalingM} of~$M$. 
For all $1\le q\le \infty$ we have,
for some constant $C'_\eta>0$,
\[
    \begin{split}
    \|I_3(t)\|_q
    &\le 
    \int_0^{a_\eta t}\!\!\int (t-s)^{-\frac{n+1}{2}}
      \bigl\| M(\textstyle\frac{\cdot\,-y}{\sqrt{t-s}},1)-M(\textstyle\frac{\cdot}{\sqrt{t-s}},1) \bigr\|_q \,|\w(y,s)|\dd y\dd s\\
  &\le  C'_\eta\, 
    t^{-\frac{1}{2}-\frac n2(1-\frac1q)}   \int_0^\infty\!\!\int 
      {\bf 1}_{[0,a_\eta t]}(s)\bigl\| M(\cdot-\textstyle\frac{y}{\sqrt{t-s}},1)-M(\cdot,1)\bigr\|_q \,|\w(y,s)|\dd y\dd s.\\
  \end{split}
  \]
The integrand is dominated by the integrable function $2\|M(\cdot,1)\|_q\,|\w(y,s)|$. Moreover, for a.e. $(y,s)\in \R^n\times(0,a_\eta t)$,
by the continuity under translations of the $L^q$-norm
(or, when $q=\infty$, by the fact that $M(\cdot,t)$ is uniformly continuous), we have
\[
   {\bf 1}_{[0,a_\eta t]}(s)\bigl\| M(\cdot-\textstyle\frac{y}{\sqrt{t-s}},1)-M(\cdot,1)\bigr\|_q \,\w(y,s)\to0, \qquad\text{as $t\to+\infty$}.
\]
The dominated convergence theorem then yields 
\[
    \|I_3\|_q=o\bigl(t^{-\frac{1}{2}-\frac n2(1-\frac{1}{q})}\bigr),
    \qquad \text{as $t\to+\infty$}.
\]

Let us now consider $I_4$. 
Applying the Young inequality
we have, for $1\le  \alpha,\beta\le q\le\infty$,
\begin{equation}
\label{eq:youn}
    \|I_4(t)\|_q
    \le \int_{a_\eta t}^t \|M(t-s)\|_\alpha 
    \|\w(\cdot,s)\|_\beta\dd s, \qquad \textstyle1+\frac1q
    =\frac1\alpha+\frac1\beta.
\end{equation}

Let us first consider the case $1\le q<\frac{n}{n-1}$.
This ensures $\frac{n}{2}(1-\frac{1}{q})<\frac{1}{2}$.
For any $\eta>0$ small enough, take $a_\eta=1-\eta^{1/(1/2-(n/2)(1-1/q))}$, in a such way that $1/2<a_\eta<1$.
In this case we apply the above estimate with $\alpha=q$ and $\beta=1$.
We now make use of the assumption $\|\w(t)\|_1=\mathcal{O}(\frac1t)$ to deduce the estimate, 
for large enough~$t$, noting that $t/2<a_\eta t$,
\[
\begin{split}
    \|I_4\|_q  
    &\le C_{q} \int_{a_\eta t}^t (t-s)^{-\frac{1}{2}-\frac n2(1-\frac1q)}s^{-1} \dd s\\
    &\le C_q\, t^{-1} \int_{a_\eta t}^t (t-s)^{-\frac{1}{2}-\frac n2(1-\frac1q)}\dd s\\
    &=C_q\,\eta\, t^{-\frac{1}{2}-\frac n2(1-\frac1q)}.
\end{split}
 \]
As $\eta>0$ is arbitrarily small, the conclusion follows in this case.

When $\frac{n}{n-1}\le q\le\infty$, we need to take in estimate~\eqref{eq:youn}
$1-\frac1n<\frac{1}{\alpha}\le1$, and so $\frac1q\le\frac{1}{\beta}<\frac1q+\frac1n$.
We consider again an arbitrary small $\eta>0$, but we take now
$a_\eta=1-\eta^{1/(1/2-(n/2)(1-1/\alpha))}$.
The additional assumption 
$\|\w(t)\|_\beta=\mathcal{O}(t^{-(1+\frac n2(1-\frac1\beta))})$ yields the estimate, for large enough~$t$,
\[
\begin{split}
    \|I_4\|_q  
    &\le C_{q} \int_{a_\eta t}^t (t-s)^{-\frac{1}{2}-\frac n2(1-\frac1\alpha)}
        s^{-1-\frac{n}{2}(1-\frac1\beta)} \dd s\\
    &\le C_q\, t^{-1-\frac{n}{2}(1-\frac1\beta)} 
    \int_{a_\eta t}^t (t-s)^{-\frac{1}{2}-\frac n2(1-\frac1\alpha)}\dd s\\
    &=C_q\,\eta\, t^{-\frac{1}{2}-\frac n2(1-\frac1q)}.
\end{split}
 \]
Hence, the conclusion follows also in this case.
\end{proof}

\section{Rapidly dissipative heat flows}
\label{sec:hea}

The goal of this section is to state a simple variant of the well known
characterization of the Besov space $\dot B^{-s}_{r,\infty}(\R^n)$,
with $s>0$, in terms of the heat kernel: such a characterization asserts that 
$f\in \dot B^{-s}_{r,\infty}(\R^n)$ if and only if $\|e^{t\Delta} f\|_r=
\mathcal{O}(t^{-s/2})$, the Besov norm $\|f\|_{\dot B^{-s}_{r,\infty}}$ being equivalent to $\sup_{t>0}t^{s/2}\|e^{t\Delta}f\|_r$. See \cite[Chap.~2]{BCD}.
In the same way, the subspace $\dot B^{-s}_{r,c_0} (\R^n)$ characterizes the
tempered distributions $f$ such that  $\|e^{t\Delta} f\|_r=
o(t^{-s/2})$ as $t\to+\infty$.

The lemma below will be used to establish the last claim
of Theorem~\ref{th:theo1}.

\begin{lemma}
\label{lem:hf}
Let $1\le r\le\infty$, $s>0$ and $f\in \dot B^{-s}_{r,\infty}(\R^n)$. 
Then $f$ belongs to $\dot B^{-s}_{r,c_0}(\R^n)$ if and only if 
$\lim_{t\to+\infty} t^{s/2}\|e^{t\Delta}f\|_r=0$.
\end{lemma}

\begin{proof}
The assertion can be proved following the steps of \cite[Theorem~2.34]{BCD}, so we highlight only the relevant modifications.
In the following calculations, $c,C$ are some positive constants, possibly
dependent on the fixed parameters $s$ and $r$ (but not on $\epsilon$ and $j$).
We use the Littlewood-Paley decomposition of $f$ and use the known
estimate (see \cite[p.73]{BCD}):
\[
\begin{split}
t^{s/2} \|e^{t\Delta}f\|_r 
&\le \sum_{j\in\Z} \|t^{s/2}\Delta_j e^{t\Delta} f\|_r 
\le C\|f\|_{\dot B^{-s}_{r,\infty}} 
\sum_{j\in\Z} t^{s/2} 2^{js} e^{-ct2^{2j}} c_{j},
\end{split}
\]
where $|c_j|\le1 $. When $f\in \dot B^{-s}_{r,c_0}$, then we have $c_j\to0$ as $j\to-\infty$.
So we just have to prove that
\[
\lim_{t\to+\infty}\sum_{j\in\Z} \Bigl(t^{s/2} 2^{js} e^{-ct2^{2j}} c_{j}\Bigr)=0.
\]
Indeed, let $\epsilon>0$ and $j_\epsilon\in \Z$, such that if
$j\le j_\epsilon$ then $|c_j|<\epsilon^{1+s/2}$.
There exists $t_\epsilon>0$ such that, if $t\ge t_\epsilon$,
then $2^{j_\epsilon}>(\epsilon t)^{-1/2}$.
Therefore, for any $t\ge  t_\epsilon$, on one hand,
\[
\sum_{2^{j}\le (\epsilon t)^{-1/2}} 
\Bigl(t^{s/2} 2^{js} e^{-ct2^{2j}} |c_{j}|\Bigr)
\le 
\sum_{2^{j}\le (\epsilon t)^{-1/2}} t^{s/2}2^{js} \epsilon^{1+s/2}
\le C\epsilon.
\]
On the other hand,
\[
\sum_{2^{j}\ge (\epsilon t)^{-1/2}} 
\Bigl(t^{s/2} 2^{js} e^{-ct2^{2j}} |c_{j}|\Bigr)
\le 
\sum_{2^{j}\ge (\epsilon t)^{-1/2}} (t2^{2j})^{s/2}e^{-ct2^{2j}}
\le 
\sum_{2^{j}\ge (\epsilon t)^{-1/2}} \frac{C}{t2^{2j}}
\le 2C\epsilon.
\]
This in turn implies our assertion, and so $t^{s/2}\|e^{t\Delta}f\|_r\to0$
as $t\to+\infty$ when $f\in \dot B^{-s}_{r,c_0}(\R^n)$.

The converse part relies on the inequality (see \cite[p.74]{BCD}).
\[
\|\Delta_j f\|_r \le C\int_0^\infty t^{s/2}2^{2j(s/2+1)} e^{-c t 2^{2j}}
\|e^{t\Delta} f\|_r\dd t.
\]
Let $\epsilon>0$. Then for any $t\ge t_\epsilon>0$ large enough, 
we have $t^{s/2}\|e^{t\Delta}f\|_r<\epsilon$. Hence, 
for some constant $C_f=\sup_{t>0} t^{s/2}\|e^{t\Delta}f\|_r$,
\[
\begin{split}
\| \Delta_j f\|_r 
&\le C_f \int_0^{t_\epsilon} 2^{2j(s/2+1)} e^{-c t 2^{2j}}\dd t
+C\epsilon \int_{0}^\infty 2^{2j(s/2+1)} e^{-c t 2^{2j}} \dd t
\\
&\le C_f 2^{js} \int_0^{t_\epsilon 2^{2j}}  e^{-c t }\dd t
+C\epsilon 2^{js}.
\end{split}
\]
From this we deduce that $2^{-js}\|\Delta_j f\|_r\to0$ as $j\to-\infty$.
\end{proof}

\section{Constructing $f$ and a solution $v$ with a prescribed
energy matrix}
\label{sec:algo}

Next step of the proof of Theorem~\ref{th:theo1} is the algorithmic construction of a tensor~$f$ and of 
a solution $v$ of the Navier--Stokes equations with external force of the divergence form $\nabla\cdot f$, such that the energy matrix
\[
\int_0^\infty\!\!\!\int (v\otimes v-f)(x,t)\dd x\dd t
\]
is a scalar multiple of the identity matrix.
This is done in the following proposition.

\begin{proposition}
\label{pro:algo}
Let $2\le n<r<\infty$.
There exists a constant $\eta_0>0$ (only dependent on~$n$ and~$r$) with the following property.
Let
$
a\in 
\dot B^{-1+n/r}_{r,\infty}(\R^n)\cap \dot H^{-1}(\R^n)$, $\nabla\cdot a=0.
$
Let $\chi\in Y_r$ and $\chi\in L^1(\R^n\times\R^+)$ such that 
$\int_0^\infty\!\!\int\chi=1$.
%with $\chi\in Y_r$ 
%\begin{equation}
%\label{hypoc}
%\esup_{t>0} t\|\chi(t)\|_1
%+\esup_{t>0}t^{1+n/2}\|\chi(t)\|_\infty<\infty.
%\end{equation}
We assume that
\begin{subequations}
\begin{equation}
\label{small:alla}
\|a\|_{\dot B^{-1+n/r}_{r,\infty}}+ \|a\|_{\dot H^{-1}}^2\|\chi\|_{Y_r}<\eta_0
\end{equation}
and that at least one of the following holds:
\begin{equation}
\label{small:allb}
 \|a\|_{\dot H^{-1}}
 \|\chi\|_{L^2(\R^+;\dot H^{-1})} <\eta_0
\qquad
\text{or}
\qquad
 \|a\|_{\dot H^{-1}} \|\chi\|_{L^{\frac{4+2n}{4+n}}(\R^n\times\R^+)} <\eta_0.
\end{equation}
\end{subequations}

There exists a real matrix $(\sigma_{k\ell})$ (possibly dependent on $a$ and $\chi$), with
\[
|\sigma_{k\ell}|\le 1
\qquad(k,\ell=1,\ldots,n)
\]
such that if 
\[
f=(f_{k\ell}),\quad 
\text{with}\quad
f_{k\ell}(x,t)=\sigma_{k\ell}\,\|a\|_{\dot H^{-1}}^2\chi(x,t), 
\]
and $a$ and $\chi$ satisfy the previous conditions,
then the solution $v$ of (NS) given by Proposition~\ref{prop:exi}
satisfies, for some $c\in\R$,
\begin{equation}
\label{eq:ort}
\int_0^\infty\!\!\!\int (v_k v_\ell)(y,s)\dd y\dd s-\int_0^\infty\!\!\!\int f_{k\ell} (y,s)\dd y \dd s
=c\,\delta_{k\ell}
\qquad(k,\ell=1,\ldots,n).
\end{equation}
\end{proposition}

\begin{proof}
Let us construct, by induction, a sequence of matrix-valued functions 
$(f^{(m)})_{m\in\N}$ and a sequence of solutions $(u^{(m)})_{m\in\N}$ of the Navier-Stokes equations with external forces
of divergence form $\nabla\cdot f^{(m)}$,
\begin{equation}
\label{INSm}
\tag{NS$_m$}
u^{(m)}(t)=e^{t\Delta}a+\int_0^t e^{(t-s)\Delta}\PP\nabla\cdot f^{(m)}(s)\dd s
+G(u^{(m)},u^{(m)})(t).
\end{equation}
We start from an initial data $a\in \dot B^{-1+n/r}_{r,\infty}(\R^n)\cap \dot H^{-1}(\R^n)$, with $2\le n<r<\infty$, such that 
\begin{equation}
\label{eq:small-a}
    \|a\|_{\dot B^{-1+n/r}_{r,\infty}}<\eta
\end{equation}
 (the constant of Proposition~\ref{prop:exi}).
 
 \medskip
\noindent{\emph{First step:}}
We initialise our construction putting
\[
f^{(0)}=0.
\]
In this case, we deduce from Proposition~\ref{prop:exi} that a solution 
$u^{(0)}\in X_r\cap L^2(\R^n\times \R^+)$ to (NS$_0$) does exist,
satisfying the estimates
\[
    \|u^{(0)}\|_{X_r}< C_0\eta,\qquad
    \|u^{(0)}\|_{L^2(\R^n\times\R^+)}\le {\textstyle\frac{4}{5}}\|a\|_{\dot H^{-1}}.
\]
Now, let $m\in\{1,2,\dots\}$.
We inductively assume that there exists $f^{(m-1)}$ and  a global mild solution 
$u^{(m-1)}$ of (NS$_{m-1}$)
such that
\begin{subequations}
    \label{subab}
        \begin{align}
        \label{IA1}
        \tag{IA1}
        &u^{(m-1)}\in X_r, \qquad \qquad\qquad\quad\:
        \|u^{(m-1)}\|_{X_r}<C_0\eta\\
        \label{IA2}
        \tag{IA2}
        &u^{(m-1)}\in L^2(\R^n\times \R^+), \qquad 
        \|u^{(m-1)}\|_{L^2(\R^n\times \R^+)} \le  \|a\|_{\dot H^{-1}}.
    \end{align}
    \end{subequations}

%Notice that $\frac54<\sqrt 2$, so~\eqref{IA2} holds for $m=1$.
We construct $f^{(m)}$ using the function $\chi$. For the moment we will just need two properties of $\chi$, namely 
$\chi\in Y_r$ and $\chi\in L^2(\R^+,\dot H^{-1}(\R^n))$.
(The latter condition, could be replaced by the condition 
$\chi\in L^{(4+2n)/(4+n)}(\R^n\times\R^+)$, making some obvious modification in
the arguments below).
Next we define, for $k,\ell=1,\dots,n$,
\begin{equation}
\label{eq:recf}
f^{(m)}_{k\ell}(x,t)
:=\begin{cases}
c_{k\ell}^{(m-1)}\chi(x,t) &\; k\neq \ell, \\
(c_{kk}^{(m-1)}-\bar{c}^{(m-1)})\chi(x,t) &\;k=\ell,
\end{cases}
\end{equation}
where
\begin{equation*}
\label{compon}
c_{k\ell}^{(m-1)}
=\int_0^\infty\!\!\!\int_{\R^n} 
u^{(m-1)}_k u_\ell^{(m-1)} \text{\,d}y\text{\,d}s,
\end{equation*}
and\footnote{
The precise of choice of $\bar c^{(m-1)}$
is not important: we could choose to take 
$\bar c^{(m-1)}$ to be any real convergent sequence as $m\to+\infty$.   
For example,  
the simple choice $\bar c^{(m-1)}=0$, or  
$\bar c^{(m-1)}=\text{const.}$
would do. 
But the choice of defining $\bar{c}^{(m-1)}$ as the average
of $c_{11}^{(m-1)},\ldots, c_{nn}^{(m-1)}$ is somehow more natural for the following reason:
if there is no need to put a forcing on the fluid, because the flow is already rapidly dissipative, then our forcing term vanishes.
For example, if one starts from a symmetric initial data in the sense of~\cite{Brandolese MA 2004}
(or if for some other reason the energy matrix of the flow is a diagonal matrix),
then the corresponding solution of the free Navier--Stokes equations will be rapidly dissipative and so
there is no need to apply an external forcing to the fluid to get a rapidly dissipative flow.
With our choice of $\bar{c}^{(m-1)}$ we do have that $f^{(m)}=0$,
for any~$m$.}
\begin{equation*}
\bar{c}^{(m-1)}
=\frac{1}{n}\bigl(c_{11}^{(m-1)}+\dots+c_{nn}^{(m-1)}\bigr).
\end{equation*}

By \eqref{IA2}, $u^{(m-1)}\in L^2(\R^n\times\R^+)$, hence $f^{(m)}$ is well defined. Moreover,
\begin{subequations}
    \begin{equation}
    \label{eq:inef}
        \|f^{(m)}\|_{Y_r}
        \leq
        n^2\|a\|_{\dot H^{-1}}^2
        \|\chi\|_{Y_r}
    \end{equation}
    and
    \begin{equation}
        \|f^{(m)}\|_{L^2(\R^+;\dot H^{-1})}
        \leq n^2\|a\|_{\dot H^{-1}}^2
        \|\chi\|_{L^2(\R^+;\dot H^{-1})}.
    \end{equation}
    \end{subequations}
Now, under the smallness condition,
\begin{equation}
\label{eq:newsmall}
\|a\|_{\dot B^{-1+n/r}_{r,\infty}}+ n^2\|a\|_{\dot H^{-1}}^2
\|\chi\|_{Y_r} <\eta,
\end{equation}
assumption~\eqref{eq: lll} will be satisfied with $f^{(m)}$ instead of~$f$. 
Then, from items i) and ii) of Proposition~\ref{prop:exi},
we deduce the existence of a solution $u^{(m)}\in X_{r}$ to (NS$_m$) 
such that 
\begin{equation*}
%\label{eq:bound-u1}
\|u^{(m)}\|_{X_r}< C_0\eta,
\qquad
\|u^{(m)}\|_{X_\infty} \le C_0\eta,
\end{equation*}
and
\begin{equation}
    \label{eq:L2bound}
    \begin{split}
        \|u^{(m)}\|_{L^2(\R^n\times \R^+)} 
        &\le
        {\textstyle \frac{4}{5}}\|a\|_{\dot H^{-1}}
            +C_0\|f^{(m)}\|_{L^2(\R^+;\dot H^{-1})}\\
        &\le 
        {\textstyle \frac{4}{5}}\|a\|_{\dot H^{-1}}
            +  n^2C_0 \|a\|_{\dot H^{-1}}^2 \|\chi\|_{L^2(\R^+;\dot H^{-1})}\\
        &\le \|a\|_{\dot  H^{-1}}.
      \end{split}
    \end{equation}
For the validity of the last inequality we need to put one additional size condition on $\chi$,
namely
\begin{subequations}
\begin{equation}
\label{eq:newsmall2}
n^2C_0\|a\|_{\dot H^{-1}}
\|\chi\|_{L^2(\R^+;\dot H^{-1})}
\le \frac{1}{5}.
\end{equation}

Summarizing, if $a$ and $\chi$ satisfy conditions~\eqref{eq:newsmall} and \eqref{eq:newsmall2}, then conditions (IA1)--(IA2) hold true with $u^{(m)}$ instead of $u^{(m-1)}$.
Therefore, under such conditions the sequence $(f^{(m)})$ and the sequence of solutions $(u^{(m)})$ are indeed well defined by algorithm~\eqref{eq:recf}. Moreover, (IA1)--(IA2) hold by induction for any integer $m\ge1$.
Instead of~\eqref{eq:newsmall2}, one could otherwise work under the condition

\begin{equation}
\label{eq:newsmall2'}
    n^2C_0\|a\|_{\dot H^{-1}}\|\chi\|_{L^{\frac{4+2n}{4+n}}(\R^n\times \R^+)}\le \frac{1}{5}
\end{equation}
\end{subequations}
and obtain the same result.

\medskip
\noindent\emph{Second step : convergence.}
%\begin{lemma}
%\label{lem-sa}
%Let $2\le n<r<\infty$.
%There exist a constant $\eta'>0$ (with $\eta'<\eta$, the constant of Proposition~\ref{prop:exi}) such that if $a\in \dot B^{-1+n/r}_{r,\infty}\cap \dot H^{-1}$ and $\chi\in Y_r\cap L^2(\R^+,\dot H^{-1})$, with
%\begin{equation}
%\label{small:all}
%\begin{cases}
%\|a\|_{\dot B^{-1+n/r}_{r,\infty}}+ (4\|a\|_{\dot H^{-1}})^2\|\chi\|_{Y_n}<\eta'\\
%\|\chi_1\|_{L^2(\R^+,\dot H^{-1})}+\|\chi_2\|_{L^1(\R^+,\mathcal{F}L^1(\R^n))}< \eta'\\
% \|a\|_{\dot H^{-1}}
% \|\chi\|_{L^2(\R^+,\dot H^{-1})}
% <\eta'
%\end{cases}
%\end{equation}
Let us now establish the convergence of the sequences of solutions and
forcing terms constructed before.
The goal is to show that, under the conditions~\eqref{small:alla}-\eqref{small:allb},
there exists $f\in Y_r\cap L^2(\R^+;\dot H^{-1})$ and a mild solution 
$v\in X_r\cap L^2(\R^n\times\R^+)$ of (NS)
such that
\[
\begin{split}
 & u^{(m)} \to v \quad\text{in $L^2(\R^n\times \R^+)$ in $X_{r}$ and in $X_\infty$}\\
 & f^{(m)} \to f \quad\text{in $Y_r$ and in $L^2\bigl(\R^+; \dot H^{-1}(\R^n)\bigr)$}.
\end{split}
\]
(If, instead of  \eqref{eq:newsmall2}, the condition \eqref{eq:newsmall2'} is available,
then the convergence $f^{(m)}\to f$ would hold in $L^{(4+2n)/(4+n)}(\R^n\times\R^+)$ instead of in 
$L^2\bigl(\R^+; \dot H^{-1}(\R^n)\bigr)$.)

We start observing that
\begin{equation*}
\begin{split}
u^{(m+1)}(t)-u^{(m)}(t) &=
\int_0^t e^{(t-s)\Delta}\PP\nabla\cdot
[f^{(m+1)}-f^{(m)}](s)\text{\,d}s
\\
&\qquad +\int_0^t 
e^{(t-s)\Delta}\PP\nabla\cdot
\bigl[(u^{(m+1)}-u^{(m)})\otimes u^{(m+1)}\bigr](s)\text{\,d}s 
\\
&\qquad +
\int_0^t e^{(t-s)\Delta}\PP\nabla\cdot
\bigl[u^{(m)}\otimes(u^{(m+1)}-u^{(m)})\bigr](s)\text{\,d}s
\\
&=:\mathcal{I}_1^{(m)}(t) 
+ \mathcal{I}_2^{(m)}(t)
+\mathcal{I}_3^{(m)}(t).
\end{split}
\end{equation*}
Applying Lemma~\ref{lem:mix}, Item~ii), we obtain
\begin{equation*}
\begin{split}
\|\mathcal{I}_2\|_{L^2(\R^n\times\R^+)}+\|\mathcal{I}_3\|_{L^2(\R^n\times\R^+)}
&\lesssim \Bigl(\|u^{(m)}\|_{X_r}+\|u^{(m+1)}\|_{X_r}\Bigr)
\|u^{(m+1)}-u^{(m)}\|_{L^2(\R^n\times\R^+)}\\
&\lesssim \eta\,\|u^{(m+1)}-u^{(m)}\|_{L^2(\R^n\times\R^+)}.
\end{split}
\end{equation*}
After replacing $\eta$  by a smaller constant (this impacts the suitable value of $\eta_0$ in the statement of the proposition), still depending only on $n$ and $r$,
then computing the norms in ${L^2(\R^n\times\R^+)}$, we see that
the two last term can be absorbed by the first one. Namely,
\[
\|u^{(m+1)}-u^{(m)}\|_{L^2(\R^n\times\R^+)} \lesssim \|\mathcal{I}_1^{(m)}\|_{L^2(\R^n\times\R)}.
\]
Let us now estimate $\mathcal{I}_1^{(m)}(t)$. 
Applying Lemma~\ref{lem:L2f} we get
\begin{equation*}
\begin{split}
\|\mathcal{I}_1^{(m)}\|_{L^2(\R^n\times\R^+)}
&\lesssim 
 \|f^{(m+1)}-f^{(m)}\|_{L^2(\R^+;\dot H^{-1})}.
\end{split}
\end{equation*}

But
\begin{equation*}
\begin{split}
\|f^{(m+1)} &-f^{(m)}\|_{L^2(\R^+;\dot H^{-1})}\\
&\leq \sum_{k,\ell=1}^n 
\|f_{k\ell}^{(m+1)}-f_{k\ell}^{(m)}\|_{L^2(\R^+;\dot H^{-1})}
\\
&=\biggl(\sum_{k\neq \ell}
\bigl|c_{k\ell}^{(m)}-c_{k\ell}^{(m-1)} \bigr| 
+
\sum_{k=1}^n
\bigl|(c_{kk}^{(m)}-\bar{c}^{(m)})-
(c_{kk}^{(m-1)}-\bar{c}^{(m-1)}) \bigr| \biggr) 
\|\chi\|_{L^2(\R^+;\dot H^{-1})}\\
&\lesssim
\|a\|_{\dot H^{-1}}
\|\chi\|_{L^2(\R^+;\dot H^{-1})}
\|u^{(m)}-u^{(m-1)}\|_{L^2(\R^n\times \R^+)}.
\end{split}
\end{equation*}
In the last inequality we used~\eqref{eq:L2bound}.
From this and from the two estimates above we deduce that there is a constant $\eta_0>0$ (only dependent on the space dimension and on $r$), such that if
\begin{subequations}
\begin{equation}
 \label{eq:laste}
 \|a\|_{\dot H^{-1}}\|\chi\|_{L^2(\R^+;\dot H^{-1})} <\eta_0,
\end{equation}
then 
\begin{equation*}
%\label{deca22}
\|u^{(m+1)}-u^{(m)}\|_{L^2(\R^n\times\R^+)} \le
\textstyle \frac12 
\|u^{(m)}-u^{(m-1)}\|_{L^2(\R^n\times\R^+)}
\end{equation*}
for all $m\ge1$. 
This implies that the sequence of solutions $u^{(m)}$ converges in $L^2(\R^n\times \R^+)$.
As before, instead of~\eqref{eq:laste},
one can reach the same conclusion under the condition
\begin{equation}
\label{eq:laste'}
\|a\|_{\dot H^{-1}}\|\chi\|_{L^{\frac{4+2n}{4+n}}(\R^n\times \R^+)}<\eta_0
\end{equation}
\end{subequations}

Moreover, arguing as before, but using Lemma~\ref{lem:mix}, Item~i), next  Lemma~\ref{lem:L2f}, Item~i), we see that
\[
\begin{split}
\|u^{(m+1)}-u^{(m)}\|_{X_r} 
&\lesssim \|\mathcal{I}_1^{(m)}\|_{X_r}\\
&\lesssim 
    \|f^{(m+1)}-f^{(m)}\|_{Y_r}\\
&\lesssim
\|a\|_{\dot H^{-1}}
\|\chi\|_{Y_r}
\|u^{(m)}-u^{(m-1)}\|_{L^2(\R^n\times \R^+)}.
\end{split}
\]
Therefore,
\[
\|u^{(m+1)}-u^{(m)}\|_{X_r} 
\lesssim 2^{-m}\|a\|_{\dot H^{-1}}
\|\chi\|_{Y_r} \|u^{(1)}-u^{(0)}\|_{L^2(\R^n\times \R^+)}.
\]
This proves that $u^{(m)}$ converges also in $X_r$. A similar argument establishes
the convergence of $u^{(m)}$ in $X_\infty$.

We denote by $v$ the limit in $L^2(\R^n\times \R^+)$ and in $X_r\cap X_\infty$
of $u^{(m)}$.
As $v\in L^2(\R^n\times \R^+)$, 
we can define the limit forcing term
\begin{equation*}
f_{k\ell}(x,t)
=\begin{cases}
c_{k\ell}^{(\infty)}\chi(x,t) & k\neq \ell, \\
(c_{kk}^{(\infty)}-\bar{c}^{(\infty)})\chi(x,t) &k=\ell,
\end{cases}
\end{equation*}
where
\begin{equation*}
c_{k\ell}^{(\infty)}
=\int_0^\infty\!\!\!\int_{\R^n} 
v_k v_\ell \text{\,d}y\text{\,d}s
\quad (k,\ell=1,\dots,n),
\end{equation*}
and
\begin{equation*}
\bar{c}^{(\infty)}
=\frac1n\bigl(c_{11}^{(\infty)}+\dots+c_{nn}^{(\infty)}\bigr)
=\frac1n\int_0^\infty\!\!\!\int|v|^2\dd y\dd s.
\end{equation*}

From the convergence $u^{(m)}\to v$ in $L^2(\R^n\times \R^+)$
we see that, for any $k,\ell=1,\ldots,n$,
$c_{k\ell}^{(m)} \to c_{k\ell}^{(\infty)}$ and
$\bar{c}^{(m)}\to \bar{c}^{(\infty)}$ as $m\to \infty$.
Thus, $f_{k\ell}$ belongs to the same function spaces where $\chi$ belong to,
namely $Y_r\cap L^2\bigl(\R^+;\dot H^{-1}(\R^n)\bigr)$
(or, otherwise, under the alternative condition on $\chi$, to $Y_r\cap L^{(4+2n)/(4+n)}(\R^n\times\R^+)$)
and $f_{k\ell}^{(m)} \to f_{k\ell}$ strongly as $m\to \infty$ in such function spaces

In particular, using the convergence in $Y_r$ of the sequence $f^{(m)}$ and
Lemma~\ref{lem:L2f}, we deduce that
$\int_0^t 
 e^{(t-s)\Delta} \PP \nabla\cdot f^{(m)}(s)\text{\,d}s
\to
\int_0^t 
 e^{(t-s)\Delta} \PP\nabla\cdot f(s)\text{\,d}s$
 for all $t>0$.
The strong convergence of $u^{(m)}$ to $v$ in $L^2(\R^n\times\R^+)$ and $X_r\cap X_\infty$ allows to pass to the limit in (NS$_m$).
This implies that the limit $v$ satisfies the Navier--Stokes equations, with force 
$\nabla\cdot f$, and initial data $a$, in its integral form:
\begin{equation*}
v=e^{t\Delta}a
+\Phi(f)+G(v,v).
\end{equation*}
Using that $\int_0^\infty\!\!\!\int \chi=1$, we finally
get~\eqref{eq:ort}, with $c=\frac1n\int_0^\infty\!\!\!\int |v|\dd y\dd s$.
This establishes Proposition~\ref{pro:algo}.
\end{proof}

\begin{remark}
\label{rem:tec}
For later use, we point out that, under the
assumption of Proposition~\ref{pro:algo}, if in addition
$2\le p\le \infty$
and $\chi\in \overline Y_{2n/(n+2)}\cap\overline Y_{p}$,
then $u^{(m)}\to v$ in $\overline X_p$
and $f^{(m)}\to f$ also in $\overline Y_{2n/(n+2)}\cap \overline Y_{p}$.
Here $u^{(m)}$ and $f^{(m)}$ are the sequences of solutions and forcing terms constructed before.

To prove the above claim, we fist observe that, by interpolation, $\chi\in \overline Y_2$.
Now
we proceed in the same way as at the beginning of Step~2 of previous proof, but now using Lemma~\ref{lem:mix}, Item~iii), next  Lemma~\ref{lem:L2f}, Item~iv). Then we get
\[
\begin{split}
\|u^{(m+1)}-u^{(m)}\|_{\overline X_2}+ \|u^{(m+1)}-u^{(m)}\|_{\overline X_p} 
&\lesssim \|\mathcal{I}_1^{(m)}\|_{\overline X_2} + \|\mathcal{I}_1^{(m)}\|_{\overline X_p}\\
&\lesssim 
\|f^{(m+1)}-f^{(m)}\|_{\overline Y_{\frac{2n}{n+2}}}
+\|f^{(m+1)}-f^{(m)}\|_{\overline Y_{p}}\\
&\lesssim
\|a\|_{\dot H^{-1}}
\bigl(\|\chi\|_{\overline Y_{\frac{2n}{n+2}}}+\|\chi\|_{\overline Y_{p}}\bigr)
\|u^{(m)}-u^{(m-1)}\|_{L^2(\R^n\times \R^+)}.
\end{split}
\]
But, as we have seen before, right after~\eqref{eq:laste},
$\|u^{(m+1)}-u^{(m)}\|_{L^2(\R^n\times\R^+)} \lesssim 2^{-m}\|u^{(1)}-u^{(0)}\|_{L^2(\R^n\times\R^+)}$.
Therefore,
\[
\|u^{(m+1)}-u^{(m)}\|_{\overline X_2} + \|u^{(m+1)}-u^{(m)}\|_{\overline X_p} 
\lesssim 2^{-m}\|a\|_{\dot H^{-1}}
\bigl(\|\chi\|_{\overline Y_{\frac{2n}{n+2}}}+\|\chi\|_{\overline Y_{p}}\bigr)
\|u^{(1)}-u^{(0)}\|_{L^2(\R^n\times\R^+)},
\]
which establishes the convergence in the $\overline X_2$-norm and in the $\overline X_p$-norm of the sequence $u^{(m)}\to v$
and that of $f^{(m)}\to f$ in $\overline Y_{2n/(n+2)}\cap \overline Y_{p}$.
\end{remark}

\section{Conclusion of the proof of Theorem~\ref{th:theo1}}
\label{sec:conclusion}

\begin{proof}[Proof of Theorem~\ref{th:theo1}]

We will prove the result of Theorem~\ref{th:theo1} under more general conditions
on $\chi$ than stated in the theorem.
Indeed, we need to assume that $a$ and $\chi$ satisfy the conditions of Proposition~\ref{pro:algo}.
We need to assume also that 
\begin{equation}
\label{norL1}
\|\chi(t)\|_p=\mathcal{O}\bigl(t^{-1-\frac n2(1-\frac{1}{p})}\bigr),
\qquad\text{for all $1\le p\le\infty$ as $t\to+\infty$.}
\end{equation}
And we finally need to assume that
\begin{equation}
\label{norY}
\chi\in \overline Y_{\frac{2n}{n+2}}\cap \overline Y_\infty,
\end{equation}
so that the assertion of Remark~\ref{rem:tec} applies for any $2\le p\le\infty$.
When $t$ is large, this latter condition implies a bound on the $L^p$-norms of $\chi(t)$ that is less demanding than~\eqref{norL1}.
Of course, all these conditions on $\chi$ are satisfied if 
$\chi\in L^\infty_c(\R^n\times\R^+)$ and the size conditions of Theorem~\ref{th:theo1} hold.

Let $f=(f_{k\ell})$ be the matrix-valued forcing term constructed in Proposition~\ref{pro:algo} and be $v$ the corresponding solution of (NS) with external force $\nabla\cdot f$, so that
\[
f_{k\ell}=\biggl(\int_0^\infty\!\!\!\int v_k v_\ell- 
\frac1n\int_0^\infty\!\!\!\int |v|^2 \delta_{k\ell}\biggr)\chi
\;
\qquad(k,\ell=1,\ldots,n) .
\]

We now discuss the large time decay rate of the Lebesgue norms of $v(t)$.
For $j=1,\ldots,n$, the $j$-component of the velocity $v$ satisfy 
\begin{equation*}
v_j(x,t)-e^{t\Delta}a_j(x)=
-\sum_{\ell,k=1}^n \int_0^t \!\!\!\int F_{k\ell,j}(x-y,t-s)
(v_k v_\ell-f_{k\ell})(y,s)\text{\,d}y\text{\,d}s.
\end{equation*}
Let us apply Theorem~\ref{th:asy} to each one of the terms of the summation
in the right-hand side, with 
\[
M=F_{k\ell,j}
\qquad\text{and}\qquad 
\w_{k\ell}=v_k v_\ell-f_{k\ell}
\qquad
(j,k,\ell=1,\ldots,n).\]
The required conditions on~$M$~\eqref{ass:M} do hold, by the properties
of~$F$ that we recalled at the beginning of the proof of Lemma~\ref{lem:L2f}. 

Let us check the needed conditions on~$\w$.
We have $v\in L^2(\R^n\times\R^+)$ and $\chi\in L^1(\R^n\times\R^+)$,
hence $f\in L^1(\R^n\times\R^+)$ and so $\w$ does belong to $L^1(\R^n\times\R^+)$.
Moreover, the conditions that we put on $a$ and $\chi$ insure, by Remark~\ref{rem:tec}, that $v\in \overline X_p$
for any $2\le p\le \infty$. In particular,
$\|v(t)\|_2^2={\cal O}(t^{-1})$. Combining this with condition~\eqref{norL1}
we deduce that $\|\w(t)\|_1={\cal O}(t^{-1})$.
Then Theorem~\ref{th:asy} applies
and we get, at least for $1\le q<n/(n-1)$,
\begin{equation}
\label{eq:as}
\biggl\|v_j(x,t)-e^{t\Delta}a_j(x)+ \sum_{\ell,k=1}^n F_{k\ell,j}(\cdot,t)
\int_0^\infty\!\!\!\int \w_{k\ell}\biggr\|_q
=o\bigl(t^{-\frac{1}{2}-\frac{n}{2}(1-\frac{1}{q})}\bigr).
\end{equation}
Let us extend the range of the parameter $q$ for the above asymptotic profile.
For any $1\le q\le \infty$, we have also $v\in \overline X_{2q}$. Hence, using~\eqref{norL1},
\[
\|\w(t)\|_{q}\le \|v(t)\|_{2q}^2+\|f(t)\|_{q}\lesssim t^{-1-\frac{n}{2}(1-\frac{1}{q})}.
\]
By the last assertion of~Theorem~\ref{th:asy}, we now deduce that~\eqref{eq:as}
holds true for any $1\le q\le \infty$.

But
\[
\int_0^\infty\!\!\!\int \w_{k\ell}=
\Bigl(\frac{1}{n}\int_0^\infty\!\!\!\int|v|^2\Bigr)\delta_{k\ell},
\]
i.e., the matrix $\int_0^\infty\!\!\!\int \w$ is a multiple of the identity matrix.
Hence, by Miyakawa--Schonbek criterion~\cite[Proposition 2.1]{Miyakawa Schonbek},
 \[
 \sum_{\ell,k=1}^n F_{k\ell,j}(\cdot,t)
\int_0^\infty\!\!\!\int \w_{k\ell}\equiv0,
 \]
and this in turn implies
\begin{equation}
\label{eq74}
\bigl\|v_j(x,t)-e^{t\Delta}a_j(x) \bigr\|_q
=o\bigl(t^{-\frac{1}{2}-\frac{n}{2}(1-\frac{1}{q})}\bigr)
\qquad (1\le q\le \infty).
\end{equation}

Now, if $a\in \dot B^{-(n+2)/2}_{2,c_0}(\R^n)$, then $\|e^{t\Delta}a\|_2=o(t^{-(n+2)/4})$ by Lemma~\ref{lem:hf} and we readily get $\|v(t)\|_2=o(t^{-(n+2)/4})$. Conversly, if $v$ is rapidly dissipative then 
$\|e^{t\Delta}a\|_2=o(t^{-(n+2)/4})$ by~\eqref{eq74}.
But since we are assuming also $a\in \dot H^{-1}(\R^n)$, then $\|e^{t\Delta} a\|_2\lesssim\|a\|_{\dot H^{-1}}t^{-1/2}$, and so
$\sup_{t>0} t^{(n+2)/4}\|e^{t\Delta}a\|_2<\infty$.
Then 
$a\in \dot  B^{-(n+2)/2}_{2,\infty}(\R^n)$ and applying the converse part of Lemma~\ref{lem:hf} we deduce that, in fact,  $a\in \dot B^{-(n+2)/2}_{2,c_0}(\R^n)$.
\end{proof}

Corollary~\ref{cor:uni} is an immediate consequence of Theorem~\ref{th:theo1}.

\section{Acknowledgements}

The authors would like to thank the Reviewers for their useful remarks. In particular,  
the authors gratefully acknowledge the comments of one of the Reviewers, 
that they incorporated at the end of Section~\ref{sec:statement}, 
to improve the readability of the proof of the main result.

The work of the second author is partially supported by JSPS Grant-in-Aid for Scientific Research(C) 22K03385.

%%%%%%%%%%%%%%%%%%%%%%%%%


\begin{thebibliography}{99}

\bibitem{Brandolese MA 2004}
	L. Brandolese,
	\textit{Space-time decay of Navier-Stokes flows invariant under rotations},
	Math. Ann.,
	\textbf{329} (2004),
	685--706.

\bibitem{BraOka}
   {L. Brandolese},
   {T. Okabe},
   \textit{Annihilation of slowly-decaying terms of Navier-Stokes flows by
   external forcing},
   {Nonlinearity},
   {\bf 34},
   {(2021)},
   N.{3},
   {1733--1757.}
   %issn={0951-7715},
   %doi={10.1088/1361-6544/abdbbf},


\bibitem{BraS}
   {L. Brandolese},
   {M.E. Schonbek},
   \textit{Large time behavior of the Navier-Stokes flow},
   {Handbook of mathematical analysis in mechanics of viscous
      fluids}, {Springer},
   (2018), 579--645.


\bibitem{BRA-Lon}
   {L.~Brandolese},
   \emph{On a non-solenoidal approximation to the incompressible
   Navier-Stokes equations},
   {J. Lond. Math. Soc. (2)},
   {\bf 96},
   {2017},
   {N.2},
   {326--344}.

\bibitem{BCD}
   {H. Bahouri},
   {J.-Y. Chemin},
   {R. Danchin},
   \emph{Fourier analysis and nonlinear partial differential equations},
   {Grundlehren der mathematischen Wissenschaften [Fundamental
   Principles of Mathematical Sciences]},
   {\bf 343},
   {Springer, Heidelberg},
   {2011}.
   %isbn={978-3-642-16829-1},
   %DOI 10.1007/978-3-642-16830-7,

\bibitem{Can}
   {M. Cannone},
   \textit{Ondelettes, paraproduits et Navier-Stokes},
  {Diderot Editeur, Paris},
   {1995}.
   %isbn={2-84134-021-X},

	
\bibitem{Carpio}
   A. Carpio,
   \emph{Large-time behavior in incompressible Navier-Stokes equations},
   {SIAM J. Math. Anal.},
   \textbf{27},
   {(1996)},
   N.2,
   {449--475},
	
\bibitem{Fujigaki Miyakawa SIAM}
	Y. Fujigaki and T. Miyakawa,
	\textit{Asymptotic profiles of nonstationary incompressible
              {N}avier-{S}tokes flows in the whole space},
	{SIAM J. Math. Anal.},
	\textbf{{33}} (2001),
	523--544.
	
\bibitem{Fujita Kato}
	H. Fujita and T. Kato,
	\textit{{On the {N}avier-{S}tokes initial value problem. {I}}}
	{Arch. Rational Mech. Anal.},
	\textbf{16} (1964),
	269--315.

\bibitem{GALW}
	T. Gallay, C.E.~Wayne,
	\textit{Long-time asymptotics of the Navier-Stokes and vorticity equations on $\R^3$},
	Phil. Trans. Roy. Soc. Lond.
	\textbf{360} (2002),  
	{2155-2188}.
	 
\bibitem{HLZZ}
	T. Hagstrom, J. Lorenz, J.P. Zingano, P.R. Zingano,
	\emph{On two new inequalities for Leray solutions of the Navier-Stokes equations in $\R^n$},
	 J. Math. Anal. Appl. \textbf{483}, N. 1, (2020).

\bibitem{Kato}
	T. Kato,
	\textit{Strong {$L^{p}$}-solutions 
	of the {N}avier-{S}tokes equation
              in {${\bf R}^{m}$}, with applications to weak solutions},
	{Math. Z.},
	\textbf{187}(1984),
	471--480.

\bibitem{Lem}
   {P.-G. Lemari\'{e}-Rieusset},
   \emph{Recent developments in the Navier-Stokes problem},
   Series {Chapman \& Hall/CRC Research Notes in Mathematics},
   {\bf 431},
   {Chapman \& Hall/CRC, Boca Raton, FL},
   {2002},
   ISBN{1-58488-220-4}.
   %doi={10.1201/9781420035674},

\bibitem{Lem21}
   {P.-G. Lemari\'{e}-Rieusset},
   \emph{The Navier-Stokes problem in the 21st century},
   {CRC Press, Boca Raton, FL},
   {2016},
   %isbn={978-1-4665-6621-7},

		
	
\bibitem{Miyakawa 2002 FE}
	T. Miyakawa,
	\textit{Notes on space-time decay properties of nonstationary
              incompressible Navier-Stokes flows in $\R^n$},
	{Funkcial. Ekvac.},
	\textbf{45} (2002), 
	271--289.


    
\bibitem{Miyakawa Schonbek}
    T. Miyakawa and M. E. Schonbek,
    \textit{On optimal decay rates for weak solutions 
        to the Navier-Stokes equations in $\R^n$},
    in ``Proceedings of Partial Differential Equations 
        and Applications (Olomouc, 1999),''
    Math. Bohem.,
    \textbf{126} (2001), 
    443--455.

\bibitem{Okabe Tsutsui JDE}
	T. Okabe and Y. Tsutsui,
	\textit{Navier-Stokes flow in the weighted Hardy space
	with applications to time decay problem},
	J. Differential Equations,
	\textbf{261} (2016),
	1712--1755.    

\bibitem{Schonbek 1985}
    M. E. Schonbek,
    \textit{$L^2$ decay for weak solutions 
        of the Navier-Stokes equations},
    Arch. Rational Mech. Anal.,
    \textbf{88} (1985), 
    209--222.

\bibitem{Tsutsui JFA}
	Y. Tsutsui,
	\textit{An application of 
		weighted Hardy spaces 
		to the Navier-Stokes equations},
	{J. Funct. Anal.},
	\textbf{266} (2014),
	1395--1420.

\bibitem{Vila}
   {Q. Vila},
   \emph{Time-asymptotic study of a viscous axisymmetric fluid without
   swirl},
   {J. Math. Fluid Mech.},
   \textbf{24},
  {(2022)}, N. {3},
   {Paper No. 84, 36},

\bibitem{Wiegner}
    M. Wiegner,
    \textit{Decay results for weak solutions 
        of the Navier-Stokes equations on ${\bf R}^n$},
    J. London Math. Soc. (2),
    \textbf{35} (1987),
    303--313.

\bibitem{XuZ}
   {L. Xu},
   {P. Zhang},
   \emph{Enhanced dissipation for the third component of 3D anisotropic
   Navier-Stokes equations},
   {J. Differential Equations},
   \textbf{335},
   {(2022)},
   {464--496}.

\end{thebibliography}
\end{document}